\documentclass[12pt, oneside]{amsart}

\usepackage[utf8]{inputenc}
\usepackage[T1]{fontenc}

\usepackage{amssymb}
\usepackage{amsthm}
\usepackage{amsmath}
\usepackage{mathrsfs}    
\usepackage[new]{old-arrows}   

\usepackage[a4paper,left=2cm,top=3cm,right=2cm,bottom=2cm]{geometry}

\usepackage[all,cmtip]{xy}

\usepackage{enumerate}

\usepackage{indentfirst}

\usepackage{hyperref}
\hypersetup{colorlinks=true,allcolors=black}  

\usepackage{pstricks}
\usepackage{graphicx}

\usepackage{multicol}

\DeclareMathOperator{\codim}{codim}
\DeclareMathOperator{\coker}{coker}
\DeclareMathOperator{\image}{im}

\DeclareMathOperator{\rank}{rank}

\DeclareMathOperator{\sym}{S}

\DeclareMathOperator{\tor}{Tor}

\mathchardef\ordinarycolon\mathcode`\:
\mathcode`\:=\string"8000
\begingroup \catcode`\:=\active
  \gdef:{\mathrel{\mathop\ordinarycolon}}
\endgroup

\begin{document}

\title{Equivariant basic cohomology of singular Riemannian foliations}

\author{Francisco C.~Caramello Jr.}
\address{Departamento de Matemática, Universidade Federal de Santa Catarina, R. Eng. Agr. Andrei Cristian Ferreira, 88040-900, Florianópolis - SC, Brazil}
\email{francisco.caramello@ufsc.br}


\newenvironment{proofoutline}{\proof[Proof outline]}{\endproof}
\theoremstyle{definition}
\newtheorem{example}{Example}[section]
\newtheorem{definition}[example]{Definition}
\newtheorem{remark}[example]{Remark}
\theoremstyle{plain}
\newtheorem{proposition}[example]{Proposition}
\newtheorem{theorem}[example]{Theorem}
\newtheorem{lemma}[example]{Lemma}
\newtheorem{corollary}[example]{Corollary}
\newtheorem{claim}[example]{Claim}
\newtheorem{conjecture}[example]{Conjecture}
\newtheorem{thmx}{Theorem}
\renewcommand{\thethmx}{\Alph{thmx}} 
\newtheorem{corx}{Corollary}
\renewcommand{\thecorx}{\Alph{corx}} 

\newcommand{\dif}[0]{\mathrm{d}}
\newcommand{\od}[2]{\frac{d #1}{d #2}}
\newcommand{\pd}[2]{\frac{\partial #1}{\partial #2}}
\newcommand{\dcov}[2]{\frac{\nabla #1}{d #2}}
\newcommand{\proin}[2]{\left\langle #1, #2 \right\rangle}
\newcommand{\f}[0]{\mathcal{F}}
\newcommand{\g}[0]{\mathcal{G}}
\newcommand{\metric}{\ensuremath{\mathrm{g}}}

\begin{abstract}
We extend the notion of equivariant basic cohomology to singular Riemannian foliations with transverse infinitesimal actions, aiming the particular case of singular Killing foliations, which admit a natural transverse action describing the closures of the leaves. This class of foliations includes those coming from isometric actions, as well as orbit-like foliations on simply connected manifolds. This last fact follows since we establish that the strong Molino conjecture holds for orbit-like foliations. In the spirit of the classical localization theorem of Borel and its later generalization to regular Killing foliations, we prove that the equivariant basic cohomology of a singular Killing foliation localizes to the set of closed leaves of the foliation, provided this set is well behaved. As applications, we obtain that the basic Euler characteristic also localizes to this set, and that the dimension of the basic cohomology of the localized foliation is less than or equal to that of the whole foliation, with equality occurring precisely in the equivariantly formal case.
\end{abstract}

\maketitle
\setcounter{tocdepth}{1}
\tableofcontents

\section{Introduction}
Regular Riemannian foliations are relatively well known and have a robust structural theory, due mainly to P.~Molino \cite{molino}. This theory establishes that the leaf closures of such a foliation ${\mathcal{F}}$ form a singular Riemannian foliation $\overline{{\mathcal{F}}}$, which moreover is described by the action of a locally constant sheaf $\mathscr{C}_{\mathcal{F}}$ of Lie algebras of germs of transverse Killing vector fields (see Section \ref{section: (Singular) Riemannian foliations} for more details). When this sheaf is constant we say that ${\mathcal{F}}$ is Killing, following \cite{mozgawa}. One of the most elementary algebraic invariants of a foliation is its basic cohomology $H({\mathcal{F}})$, which can be seen as the De Rham cohomology of its (often quite singular) leaf space. For a regular Killing foliation one can also consider its \emph{equivariant} basic cohomology $H_{\mathfrak{a}}({\mathcal{F}})$, with respect to the transverse action of its structural algebra $\mathfrak{a}=\mathscr{C}_{\mathcal{F}}(M)$. This invariant, introduced in \cite{goertsches} and further studied in \cite{goertsches2}, will also capture information on the leaf closures, since $\mathfrak{a}{\mathcal{F}}=\overline{{\mathcal{F}}}$. In those papers the authors proved transverse generalizations of the Borel--Hsiang and Atiyah--Bott/Berline--Vergne localization theorems, which were later extended in \cite{lin} to the case of infinitesimal transverse actions on regular Riemannian foliations in general. Moreover, in \cite{caramello2} the authors showed that the ring structure of the equivariant basic cohomology remains constant under regular deformations (see also \cite{caramello}), which linked this object to the theory of torus actions on orbifolds.

The theory of \emph{singular} Riemannian foliations also started with Molino's work, but is still in plain development, having seen relatively recent answers to fundamental questions and new interesting examples (e.g., \cite{alex4}, \cite{mendes}, \cite{lytchak}). For instance, Molino's conjecture that the closures of the leaves of a singular Riemannian foliation form another singular Riemannian foliation was established recently in \cite{alex4}. These foliations appear naturally in the theories of isometric actions and submanifolds (for example, as the set of all parallel submanifolds of an isoparametric submanifold \cite{thorbergsson}), and play interesting roles in many other areas of geometry, for instance in the study of the regularity of the Sharafutdinov projection onto the soul \cite{wilking}. In this paper we are mostly interested in singular Killing foliations, a subclass that includes foliations given by the orbits of isometric actions (see Example \ref{example: homogeneous Riemannian foliations are Killing}), as well as some Riemannian foliations on simply connected manifolds. This last fact is related to an important question, know as the \emph{strong} Molino conjecture (see \cite[p. 215]{molino}, \cite{alex2}): is the Molino sheaf of a singular Riemannian foliation smooth? We prove that this holds for the so-called orbit-like foliations (see Section \ref{section: orbit-like foliations} for detais):
\begin{thmx}\label{theoremA}
If ${\mathcal{F}}$ is orbit-like, then $\mathscr{C}_{\mathcal{F}}$ is smooth. In particular, any orbit-like foliation ${\mathcal{F}}$ of a simply connected manifold is Killing.
\end{thmx}
This is Theorem \ref{proposition: strong molino conjecture for orbit-like} in the text (see also Example \ref{orbit-like on 1connected are killing}). By analogous reasoning, if the strong Molino conjecture holds true in general, the class of singular Killing foliations will include all complete Riemannian foliations on simply connected manifolds. Theorem \ref{theoremA} hence further motivates the study of singular Killing foliations as a relevant class of Riemannian foliations, intimately connected with the study of their leaf closures.

As in the regular case, the Molino sheaf $\mathscr{C}_{\mathcal{F}}$ of a singular Killing foliation is constant, and the structural algebra $\mathfrak{a}= \mathscr{C}_{\mathcal{F}}(M)$ acts transversely on ${\mathcal{F}}$, with $\mathfrak{a}{\mathcal{F}}=\overline{{\mathcal{F}}}$ (see Sections \ref{section: (Singular) Riemannian foliations} and \ref{section: transverse action of structural algebra on a skf} for details). It is then relevant to investigate the equivariant basic cohomology $H_{\mathfrak{a}}({\mathcal{F}})$, with respect to this $\mathfrak{a}$-action. In this matter, our results generalize some of those in \cite{goertsches} to the singular setting. Moreover, when specified to the regular case, they do not involve the hypothesis of ${\mathcal{F}}$ being transversely oriented. The main result is a singular transverse generalization of Borel's localization theorem \cite{borel}, which we show to hold for the subclass of \emph{neat} singular Killing foliations (see Definition \ref{definition: neat skf}). This hypothesis is necessary because, in contrast with the regular case, for a singular Killing foliation the union $M^{\mathfrak{a}}$ of all closed leaves may not be a submanifold (see Example \ref{example: propagation of non minimal closed leaves}). We denote ${\mathcal{F}}^\mathfrak{a}:=\mathcal{F}|_{M^{\mathfrak{a}}}$.
\begin{thmx}\label{theoremB}
Let $(M,{\mathcal{F}})$ be a transversely compact, neat singular Killing foliation with structural algebra $\mathfrak{a}$, and let $i:M^{\mathfrak{a}}\to M$ be the natural inclusion. Then the localized map
$$S^{-1}i^*:S^{-1}H_{\mathfrak{a}}({\mathcal{F}})\longrightarrow S^{-1}H_{\mathfrak{a}}({\mathcal{F}}^\mathfrak{a})$$
is an isomorphism, where $S=\sym(\mathfrak{a}^*)\setminus\{0\}$.
\end{thmx}
This theorem appears below as Theorem \ref{theorem: borel localization}. The technical tools to obtain this result involve the study of the equivariancy of the Molino sheaf under the so-called homothetic transformations of ${\mathcal{F}}$ (see Proposition \ref{proposition: invariance of molino sheaf under homothetic transformations}). This enables us to construct a suitable covering of $M$ by saturated open sets which equivariantly retract by deformations onto leaf closures (see Proposition \ref{proposition: good cover}), and the proof then follows usual cohomological techniques. This tool also covers a gap in the proof of \cite{wolak} that the basic cohomology of a singular Riemannian foliation is finite dimensional, when $M$ is compact (see Section \ref{section: good covers}). In fact, we prove this here under the weaker hypothesis of ${\mathcal{F}}$ being transversely compact, in Theorem \ref{theorem: basic cohomology is finite dimensional}:
\begin{thmx}\label{theoremC}
Let ${\mathcal{F}}$ be a transversely compact, complete singular Riemannian foliation. Then $\dim H({\mathcal{F}})<\infty$.
\end{thmx}

The basic Euler characteristic $\chi({\mathcal{F}})$, in particular, is well defined for any transversely compact, complete singular Riemannian foliation. In the case of neat singular Killing foliations, a consequence of Borel's localization is that this invariant localizes to the set of closed leaves:
\begin{thmx}\label{theoremD}
Let ${\mathcal{F}}$ be a transversely compact, neat singular Killing foliation with structural algebra $\mathfrak{a}$. Then
$$\chi({\mathcal{F}})=\chi({\mathcal{F}}^\mathfrak{a}).$$
\end{thmx}
This is Theorem \ref{theorem: localization of bec}, which generalizes \cite[Theorem D]{caramello} to the singular setting. Notice that it is a singular transverse generalization of the localization of the classical Euler characteristic $\chi(M)$ to the fixed point set of a torus action, since $\mathfrak{a}$ is Abelian (see Section \ref{section: (Singular) Riemannian foliations}). It follows that, for any transversely compact, singular Killing foliation $\f$, if $\chi(\f)\neq0$, then $\f$ has at least one closed leaf (see Corollary \ref{corollary: chi detects closed leaves}). Another direct consequence is that, if $\f$ has isolated closed leaves, $\chi({\mathcal{F}})$ is precisely their number (see Corollary \ref{corollary: localization of bec finite case}). The following result is also an application of the localization theorem:
\begin{thmx}\label{theoremE}
Let ${\mathcal{F}}$ be a transversely compact, neat singular Killing foliation with structural algebra $\mathfrak{a}$. Then
$$\dim H({\mathcal{F}}^\mathfrak{a})\leq \dim H({\mathcal{F}}),$$
and equality holds if, and only if, ${\mathcal{F}}$ is equivariantly formal.
\end{thmx}
This is Theorem \ref{theorem: comparison dimensions}, which generalizes \cite[Theorem 1]{goertsches}. In analogy to the classical case of Lie group actions, and more generally regular Killing foliations (see \cite{goertsches2}), a ``concrete'' version of Borel's localization theorem, in terms of a localization formula for the integral of basic equivariant forms to the set of closed leaves, is expected to hold. We intend to investigate this in a future article.

\section{Equivariant basic cohomology of singular foliations}

When a Lie algebra acts infinitesimally and transversely on a (regular) foliation, one can define its equivariant basic cohomology, which is an algebraic invariant that captures information from the action and the topology of the leaf space. This object was first studied in \cite{goertsches}, generalizing the ideas from classical equivariant cohomology of Lie group actions. In this section we will see that this notion also generalizes to the case of \emph{singular} foliations. Let us begin by fixing our notation and recalling the main notions and tools concerning smooth foliations. We work in the smooth category, so all objects are considered of differentiability class $\mathcal{C}^\infty$, unless otherwise explicitly stated.

\subsection{Preliminaries}

Let $M$ be an $n$-dimensional manifold. A \textit{singular foliation} of $M$ is a partition ${\mathcal{F}}$ of $M$ into connected, immersed submanifolds, called \textit{leaves}, such that the module $\mathfrak{X}({\mathcal{F}})$ of smooth vector fields that are tangent to the leaves is transitive on each $L\in{\mathcal{F}}$. This means that, for each $L\in{\mathcal{F}}$ and each $x\in L$, any given $v\in T_xL$ can be extended to a smooth vector field $V\in\mathfrak{X}({\mathcal{F}})$. We denote the leaf containing $x\in M$ by ${\mathcal{F}}(x)$. The tangent spaces of the leaves form a distribution of varying rank, that we denote by $T{\mathcal{F}}$. The \textit{dimension of ${\mathcal{F}}$} is defined as
$$\dim({\mathcal{F}})=\max_{L\in{\mathcal{F}}}\dim(L),$$
and its \textit{codimension} by $\codim({\mathcal{F}})=\dim(M)-\dim({\mathcal{F}})$. We say that ${\mathcal{F}}$ is \textit{closed} when all leaves of ${\mathcal{F}}$ are closed submanifolds of $M$. When all leaves have the same dimension we say that ${\mathcal{F}}$ is a \textit{regular foliation}, or simply a \textit{foliation}, of $M$. The space of leaves $M/{\mathcal{F}}$ of ${\mathcal{F}}$ is the quotient space of $M$ by the equivalence relation that identifies points in the same leaf. A subset $J\subset M$ is \textit{saturated} when $J=\pi^{-1}\pi(J)$, for $\pi:M\to M/{\mathcal{F}}$ the canonical projection. In other words, $J$ is saturated if it contains all the leaves that it intersects.

\begin{example}
An important class of singular foliations consists of those given by Lie group actions, which are called \textit{homogeneous}: if a Lie group $H$ acts smoothly on $M$, then the partition of $M$ by the connected components of the orbits of $H$ is a singular foliation ${\mathcal{F}}_H$. If all stabilizers of this action have the same dimension, in particular if the action is locally free, then ${\mathcal{F}}_H$ is a regular foliation.
\end{example}

If $(N,\g)$ is another singular foliation, we say that a smooth map $f:(M,{\mathcal{F}})\to (N,\g)$ is \textit{foliate} when it maps leaves of ${\mathcal{F}}$ into leaves of $\g$. A smooth map $f:M\times[0,1]\to N$ is a \textit{foliate homotopy} when it is foliate with respect to the product foliation of ${\mathcal{F}}$ and the trivial foliation of $[0,1]$ by points. Alternatively, $f$ is a foliate homotopy (between $f_0$ and $f_1$) if $f_t:M\ni x\mapsto f(x,t)\in N$ is foliate for each $t\in[0,1]$. If $J\subset M$ is a saturated submanifold, then a \textit{foliate deformation retraction} $f:M\times[0,1]\to M$ of $M$ onto $J$ is a foliate homotopy which is also a deformation retraction of $M$ onto $J$.

A vector field $X\in\mathfrak{X}(M)$ is \textit{foliate} when $[X,Y]\in \mathfrak{X}({\mathcal{F}})$ for any $Y\in \mathfrak{X}({\mathcal{F}})$. These are the vector fields whose flows are foliate maps. We denote the Lie algebra of foliate vector fields by $\mathfrak{L}({\mathcal{F}})$. The \textit{transverse vector fields} are the elements of $\mathfrak{l}({\mathcal{F}}):=\mathfrak{L}({\mathcal{F}})/\mathfrak{X}({\mathcal{F}})$. Notice that $\mathfrak{l}({\mathcal{F}})$ inherits the Lie bracket from $\mathfrak{L}({\mathcal{F}})$ and is, thus, also a Lie algebra. A transverse vector field $X$ is \textit{tangent} to a saturated submanifold $J\subset M$ when it admits a representative $\tilde{X}\in\mathfrak{L}({\mathcal{F}})$ which is tangent to $J$. Notice that in this case any representative of $X$ is tangent to $J$. Let $f:(M,{\mathcal{F}})\to (N,\g)$ be foliate, $X\in\mathfrak{l}({\mathcal{F}})$ and $Y\in\mathfrak{l}(\g)$. We say that $X$ and $Y$ are \textit{$f$-related} when they admit $f$-related foliate representatives. In the case of a foliate map $f:(M,{\mathcal{F}})\to (M,{\mathcal{F}})$, we say that $X\in\mathfrak{l}({\mathcal{F}})$ is $f$-invariant when it is $f$-related to itself.

A tensor field $\xi\in\mathcal{T}^k(M)$ is called \textit{basic} when it satisfies $\iota_X\xi=0$ and $\mathcal{L}_X\xi=0$ whenever $X\in\mathfrak{X}({\mathcal{F}})$. We can consider, in particular, basic differential forms. By Cartan's formula $\mathcal{L}_X=\iota_Xd+d\iota_X$, a form $\omega\in\Omega^k(M)$ is basic if and only if, $i_X\omega=0$ and $i_X(d\omega)=0$, for all $X\in\mathfrak{X}({\mathcal{F}})$. The space of basic $k$-forms of ${\mathcal{F}}$ is denoted by $\Omega^k({\mathcal{F}})$. Then
$$\Omega({\mathcal{F}}):=\bigoplus_{k=0}^q\Omega^k({\mathcal{F}})$$
is the \textit{algebra of basic forms} of ${\mathcal{F}}$. It is clear that $\Omega({\mathcal{F}})$ is closed under the usual exterior derivative $d$, so it is a $\mathbb{Z}$-graded differential algebra (by defining $\Omega^k({\mathcal{F}})=\{0\}$ if $k<0$).

The cohomology groups of the complex
$$\cdots \stackrel{d}{\longrightarrow} \Omega^{k-1}({\mathcal{F}}) \stackrel{d}{\longrightarrow} \Omega^k({\mathcal{F}}) \stackrel{d}{\longrightarrow} \Omega^{k+1}({\mathcal{F}}) \stackrel{d}{\longrightarrow} \cdots .$$
are the \textit{basic cohomology groups} of ${\mathcal{F}}$, denoted $H^i({\mathcal{F}})$. We denote $H({\mathcal{F}})=\bigoplus H^i({\mathcal{F}})$, which is a graded algebra with the usual exterior product. Notice that when ${\mathcal{F}}$ is the trivial foliation by points, $H({\mathcal{F}})$ reduces to the De Rham cohomology of $M$. In general, we can intuitively think of $H({\mathcal{F}})$ as the De Rham cohomology of $M/{\mathcal{F}}$. If $f:(M,{\mathcal{F}})\to(N,\g)$ is foliate, then it pulls $\g$-basic forms back to ${\mathcal{F}}$-basic forms, and thus induces a linear map $f^*:H(\g)\to H({\mathcal{F}})$. For ${\mathcal{F}}$-saturated open sets $U,V\subset M$, the short exact sequence 
$$0\longrightarrow \Omega({\mathcal{F}}{\vert}_{U\cup V})\stackrel{i_U^*\oplus i_V^*}{\longrightarrow}\Omega({\mathcal{F}}{\vert}_U)\oplus \Omega({\mathcal{F}}{\vert}_V)\stackrel{j_U^*-j_V^*}{\longrightarrow} \Omega({\mathcal{F}}{\vert}_{U\cap V})\longrightarrow 0.$$
induces a Mayer--Vietoris sequence in basic cohomology, where $i_U:U\to U\cup V$, $i_V:V\to U\cup V$, $j_U:U\cap V\to U$ and $j_V:U\cap V\to V$ are the natural inclusions.

When $\dim H({\mathcal{F}})<\infty$, we define the \textit{basic Euler characteristic} of ${\mathcal{F}}$ as
$$\chi({\mathcal{F}})=\sum_i(-1)^i\dim H^i({\mathcal{F}}).$$
Basic cohomology $H({\mathcal{F}})$ can in fact be infinite-dimensional, even for a regular foliation ${\mathcal{F}}$ on a compact manifold (see \cite{ghys2}). In Theorem \ref{theorem: basic cohomology is finite dimensional} we will see sufficient conditions for it to be finite-dimensional.

\subsection{Equivariant basic cohomology}\label{section: equivariant basic cohomology}

Our definition of the equivariant basic cohomology of a singular foliation will be a direct generalization of that in the regular case, which was introduced in \cite{goertsches} via the concept of $\mathfrak{g}^\star$-algebras. We also refer to \cite{guillemin} for a thorough introduction to the topic of equivariant cohomology. Let us begin by recalling that a \textit{$\mathfrak{g}^\star$-algebra} consists of a differential $\mathbb{Z}$-graded commutative algebra $(A, d)$ acted upon by a finite-dimensional Lie algebra $\mathfrak{g}$, that is, for each $X\in\mathfrak{g}$ there are derivations $\mathcal{L}_X:A\to A$ and $\iota_X:A\to A$, depending linearly on $X$ and of degree $0$ and $-1$ respectively, satisfying
\begin{enumerate}
    \item $\iota_X^2=0$,
    \item $\mathcal{L}_X\mathcal{L}_Y-\mathcal{L}_Y\mathcal{L}_X=\mathcal{L}_{[X,Y]}$,
    \item $\mathcal{L}_X\iota_Y+\iota_Y\mathcal{L}_X=\iota_{[X,Y]}$,
    \item $\mathcal{L}_X=d \iota_X+\iota_X d$.
\end{enumerate}
If $A$ and $B$ are $\mathfrak{g}^\star$-algebras, an algebra morphism $f:A\to B$ is a morphism of $\mathfrak{g}^\star$-algebras if it commutes, in the $\mathbb{Z}$-graded sense, with $d$, $\mathcal{L}_X$ and $\iota_X$.

Now let $(M,{\mathcal{F}})$ be a singular foliation and $\mathfrak{g}$ a finite-dimensional Lie algebra. An \textit{infinitesimal transverse action} of $\mathfrak{g}$ on ${\mathcal{F}}$ is a Lie algebra homomorphism $\mu:\mathfrak{g}\to \mathfrak{l}({\mathcal{F}})$. The isotropy $\mathfrak{g}_x=\{X\in \mathfrak{g}\ {\vert}\ \mu(X)_x=0\}$ is constant for points in the same leaf, so we define $\mathfrak{g}_L=\mathfrak{g}_x$, for $x\in L\in {\mathcal{F}}$, which is the infinitesimal transverse analog of the stabilizer of a point. More generally, for a saturated submanifold $J\subset M$, consider the \textit{isotropy subalgebra}
$$\mathfrak{g}_J=\{X\in \mathfrak{g}\ {\vert}\ \mu(X){\vert}_J=0\}=\bigcap_{x\in J} \mathfrak{g}_x.$$
Notice that $\mathfrak{g}_L=\mathfrak{g}_{\overline{L}}$, for $L\in {\mathcal{F}}$. Somewhat conversely, for $\mathfrak{h}<\mathfrak{g}$, the \textit{$\mathfrak{h}$-fixed locus}
$$M^{\mathfrak{h}}=\{x\in M\ {\vert}\ \mu(X)_x=0\ \mbox{for all}\ X\in\mathfrak{h}\}$$
is saturated (but not a submanifold in general). A saturated submanifold $J\subset M$ is \textit{$\mathfrak{h}$-invariant} when $\mu(X)$ is tangent to $J$, for each $X\in \mathfrak{h}$.

Let us prove that an infinitesimal transverse action induces a $\mathfrak{g}^\star$-algebra structure on $\Omega({\mathcal{F}})$. More precisely, for $X\in \mathfrak{g}$, consider the derivations $\mathcal{L}_X$ and $\iota_X$ on $\Omega({\mathcal{F}})$ given by $\mathcal{L}_X\omega=\mathcal{L}_{\tilde{X}}\omega$ and $\iota_{X}\omega=\iota_{\tilde{X}}\omega$, where $\tilde{X}\in\mathfrak{L}({\mathcal{F}})$ is a foliate representative of $\mu(X)$. One readily checks that they are well defined: if $\tilde{X}_1,\tilde{X}_2\in\mathfrak{L}({\mathcal{F}})$ are different representatives for $\mu(X)$, then $\tilde{X}_1-\tilde{X}_2\in\mathfrak{X}({\mathcal{F}})$, thus $\mathcal{L}_{\tilde{X}_1-\tilde{X}_2}\omega=0$ and $\iota_{\tilde{X}_1-\tilde{X}_2}\omega=0$ for any $\omega\in\Omega({\mathcal{F}})$. It is clear that $\mathcal{L}_X$ and $\iota_X$ have degree $0$ and $-1$, respectively.

\begin{proposition}\label{proposition: g star algebra stricture on basic complex}
Let $(M,{\mathcal{F}})$ be a singular foliation and $\mu:\mathfrak{g}\to \mathfrak{l}({\mathcal{F}})$ an infinitesimal transverse action. Then, with the derivations $\mathcal{L}_X$ and $\iota_{X}$ defined above, $(\Omega({\mathcal{F}}),d)$ becomes a $\mathfrak{g}^\star$-algebra.
\end{proposition}

\begin{proof}
The algebraic properties that $\mathcal{L}_X$, $\iota_X$ and $d$ must satisfy all follow directly from the same properties for the usual Lie derivative, interior multiplication and exterior derivative. We thus only have to show that $\mathcal{L}_X\omega$ and $\iota_X\omega$ indeed remain in $\Omega({\mathcal{F}})$ when $\omega\in\Omega({\mathcal{F}})$. In fact, for $Y\in\mathfrak{X}({\mathcal{F}})$, we have $\iota_Y\iota_{\tilde{X}}\omega=-\iota_{\tilde{X}}\iota_Y\omega=0$ (because $\iota_Y\omega=0$) and $\mathcal{L}_Y\iota_{\tilde{X}}\omega=\iota_{\tilde{X}}\mathcal{L}_Y\omega+\iota_{[Y,\tilde{X}]}\omega=0$ (because $\mathcal{L}_Y\omega=0$ and $[Y,{\tilde{X}}]\in\mathfrak{X}({\mathcal{F}})$), so $\iota_X\omega\in\Omega({\mathcal{F}})$. Similarly, $\iota_Y\mathcal{L}_{\tilde{X}}\omega=\mathcal{L}_{\tilde{X}}\iota_Y\omega-\iota_{[\tilde{X},Y]}\omega=0$ and $\mathcal{L}_Y\mathcal{L}_{\tilde{X}}\omega=\mathcal{L}_{\tilde{X}}\mathcal{L}_Y\omega-\mathcal{L}_{[\tilde{X},Y]}\omega=0$, hence $\mathcal{L}_X\omega\in\Omega({\mathcal{F}})$.
\end{proof}

Given Proposition \ref{proposition: g star algebra stricture on basic complex}, we can now define the equivariant basic cohomology of ${\mathcal{F}}$ in the usual Cartan's model way, which we recall here for the reader's convenience (see also \cite{guillemin} or \cite{goertsches}). Let $\sym(\mathfrak{g}^*)$ denote the symmetric algebra over the dual $\mathfrak{g}^*$, which we interpret as the polynomial algebra on $\mathfrak{g}^*$. Consider the coadjoint action of $\mathfrak{g}$ on $\mathfrak{g}^*$ given by $(\mathrm{ad}^*_{X}\phi)(Y)=\phi(-[X,Y])$, for $X,Y\in\mathfrak{g}$ and $\phi\in\mathfrak{g}^*$. It extends naturally to $\sym(\mathfrak{g}^*)$ by $(\mathrm{ad}^*_{X}f)(Y)=f(-[X,Y])$. For $X\in\mathfrak{g}$ and $\omega=\sum f_i\otimes \omega_i \in \sym(\mathfrak{g}^*)\otimes \Omega({\mathcal{F}})$, we define
$$X\omega=\sum \left(\mathrm{ad}^*_{X}f_i\otimes \omega_i + f_i\otimes \mathcal{L}_X\omega_i\right).$$
The \textit{basic Cartan complex of $({\mathcal{F}},\mathfrak{g})$} is the subspace
$$C_{\mathfrak{g}}({\mathcal{F}})=(\sym(\mathfrak{g}^*)\otimes \Omega({\mathcal{F}}))^{\mathfrak{g}}$$
of $\mathfrak{g}$-invariant elements $\omega$, that is, those satisfying $X\omega=0$. Notice that an element $\omega\in C_{\mathfrak{g}}({\mathcal{F}})$ can be seen as a $\mathfrak{g}$-equivariant polynomial map $\mathfrak{g}\to \Omega({\mathcal{F}})$. Moreover, if $\mathfrak{g}$ is Abelian then its coadjoint action is trivial, so $\omega\in C_{\mathfrak{g}}({\mathcal{F}})$ is just a polynomial map $\mathfrak{g}\to \Omega({\mathcal{F}})^\mathfrak{g}$.

The basic Cartan complex $C_{\mathfrak{g}}({\mathcal{F}})$ is naturally an $\sym(\mathfrak{g}^*)^{\mathfrak{g}}$-algebra, with ring multiplication $(\omega\eta)(X)=\omega(X)\wedge\eta(X)$ and $\sym(\mathfrak{g}^*)^{\mathfrak{g}}$-module multiplication induced by $f\mapsto f\otimes 1$.  Moreover, it is a cochain complex with coboundary maps $d_\mathfrak{g}$, the \textit{equivariant derivative}, given by
$$(d_\mathfrak{g}\omega)(X)=d(\omega(X))-\iota_X(\omega(X)),$$
as one checks that this is a degree $1$ derivation with respect to the usual grading
\begin{equation}\label{grading of the Cartan complex}
C_{\mathfrak{g}}^n(\Omega({\mathcal{F}}))=\bigoplus_{2k+l=n}(\sym^k(\mathfrak{g}^*)\otimes \Omega^l({\mathcal{F}}))^{\mathfrak{g}}.
\end{equation}
We define the \textit{$\mathfrak{g}$-equivariant basic cohomology of ${\mathcal{F}}$} as the cohomology of its Cartan complex:
$$H_\mathfrak{g}({\mathcal{F}})=H(C_\mathfrak{g}(\Omega({\mathcal{F}}),d_\mathfrak{g})).$$
It inherits an $\sym(\mathfrak{g}^*)^{\mathfrak{g}}$-algebra structure from $C_{\mathfrak{g}}({\mathcal{F}})$. When $\mathfrak{g}=0$ is the trivial Lie algebra, $H_{\mathfrak{g}}({\mathcal{F}})$ is just the usual basic cohomology of ${\mathcal{F}}$. Another simple, but important particular case, is the following.

\begin{example}\label{example: equivariant cohomology of trivial action}
Suppose the transverse action of $\mathfrak{g}$ on ${\mathcal{F}}$ is trivial. Then one easily sees that $C_{\mathfrak{g}}({\mathcal{F}})=\sym(\mathfrak{g}^*)^{\mathfrak{g}}\otimes \Omega({\mathcal{F}})$ and $(d_\mathfrak{g}\omega)(X)=d(\omega(X))$, so
$$H_{\mathfrak{g}}({\mathcal{F}})\cong \sym(\mathfrak{g}^*)^{\mathfrak{g}}\otimes H({\mathcal{F}}),$$
as $\sym(\mathfrak{g}^*)^{\mathfrak{g}}$-algebras.
\end{example}

\subsection{Homotopy invariance}

Let us now study the invariance of equivariant basic cohomology under homotopies, so let $(M,{\mathcal{F}})$ and $(N,\g)$ be singular foliations with transverse $\mathfrak{g}$-actions $\mu:\mathfrak{g}\to\mathfrak{l}({\mathcal{F}})$ and $\nu:\mathfrak{g}\to\mathfrak{l}(\g)$, respectively. A foliate map $f:M\to N$ is called \textit{$\mathfrak{g}$-equivariant} if $\mu(X)$ and $\nu(X)$ are $f$-related, for all $X\in\mathfrak{g}$. In this case it is easy to check that the pullback $f^*:\Omega(\g)\to\Omega({\mathcal{F}})$ is a degree $0$ morphism of $\mathfrak{g}^\star$-algebras. In fact, it is clear that it commutes with $d$, and a straightforward calculation shows that it also commutes with each $\iota_X$. The commutation with $\mathcal{L}_X$ then follows easily using Cartan's formula or also by a direct computation. Hence, we obtain an induced morphism in equivariant basic cohomology $(f^*)_*:H_\mathfrak{g}(\g)\to H_\mathfrak{g}({\mathcal{F}})$, which we will often denote simply by $f^*$ when there is no risk of confusion. A foliate homotopy $f:M\times[0,1]\to N$ is \textit{$\mathfrak{g}$-equivariant} if $f_t$ is $\mathfrak{g}$-equivariant for each $t\in[0,1]$.

We recall that, for $\mathfrak{g}^\star$-algebras $A$ and $B$, two morphisms $f_0,f_1:A\to B$ are \textit{chain homotopic} when there is a linear map $Q:A\to B$ of degree $-1$ that satisfies
$$\iota_XQ+Q\iota_X=0,\ \ \ \mathcal{L}_XQ-Q\mathcal{L}_X=0\ \ \mbox{and}\ \ dQ+Qd=f_1-f_0$$
for all $X\in\mathfrak{g}$. This implies that the induced maps on equivariant cohomology are equal, that is $(f_0)_*=(f_1)_*:H_{\mathfrak{g}}(A)\to H_{\mathfrak{g}}(B)$ (see \cite[Proposition 2.4.1]{guillemin}). The following proposition is a generalization of \cite[Lemma 2.5.1]{lin} to singular foliations, and the proof is an adaptation of a classical proof of the homotopy invariance of equivariant cohomology \cite[p. 22]{guillemin}.

\begin{proposition}[Homotopy invariance of equivariant basic cohomology]\label{proposition: Homotopy invariance of equivariant basic cohomology}
Let $(M,{\mathcal{F}})$ and $(N,\g)$ be singular foliations acted upon transversely by a Lie algebra $\mathfrak{g}$, and let $f:M\times[0,1]\to N$ be a foliate $\mathfrak{g}$-equivariant homotopy. Then
$$f_0^*=f_1^*:H_{\mathfrak{g}}(\g)\longrightarrow H_{\mathfrak{g}}({\mathcal{F}}).$$
\end{proposition}

\begin{proof}
As we saw above, the conclusion will follow if we show that $f_0^*:\Omega(\g)\to\Omega({\mathcal{F}})$ and $f_1^*:\Omega(\g)\to\Omega({\mathcal{F}})$ are chain homotopic. For each $t\in [0,1]$, consider $V_t:M\to TN$ given by
$$V_t(x)=\left.\frac{d}{dt}f(x,s)\right\vert_{s=t},$$
and for $\omega\in\Omega^{k+1}(\g)$, let $f_t^*(\iota_{V_t}\omega)\in\Omega^k(M)$ denote the form
$$f_t^*(\iota_{V_t}\omega)_x(v_1,\dots,v_k)=\omega_{f_t(x)}(V_t(x),\dif f_t(v_1),\dots,\dif f_t(v_k)).$$
One has the identity (see \cite[p. 158]{guillemin2})
\begin{equation}\label{eq: homotopy invariance}\frac{d}{dt}f_t^*\omega=f_t^*(\iota_{V_t}d\omega)+df_t^*(\iota_{V_t}\omega).\end{equation}
We claim that $f_t^*(\iota_{V_t}\omega)\in\Omega^k({\mathcal{F}})$. In fact, since each $f_t$ is foliate, $\dif f_t$ maps $T{\mathcal{F}}$ into $T\g$, and then it follows easily that $\iota_Xf_t^*(\iota_{V_t}\omega)=0$ when $X\in\mathfrak{X}({\mathcal{F}})$. Hence, as $d\omega \in\Omega^{k+2}(\g)$, we also have $\iota_Xf_t^*(\iota_{V_t}d\omega)=0$, and from equation \eqref{eq: homotopy invariance} we have
$$\iota_Xdf_t^*(\iota_{V_t}\omega)=\iota_X\frac{d}{dt}f_t^*\omega-\iota_Xf_t^*(\iota_{V_t}d\omega)=0.$$

Now notice that integrating equation \eqref{eq: homotopy invariance} we obtain
$$f_1^*\omega-f_0^*\omega=\int_0^1\frac{d}{dt}f_t^*\omega\;dt=\int_0^1f_t^*(\iota_{V_t}d\omega)\;dt+\int_0^1 df_t^*(\iota_{V_t}\omega)\;dt,$$
so if we define $Q:\Omega(\g)\to\Omega({\mathcal{F}})$ by
$$Q\omega=\int_0^1f_t^*(\iota_{V_t}\omega)\;dt,$$
then it is a degree $-1$ linear map satisfying $f_1^*-f_0^*=dQ+Qd$.

It remains to show that $Q$ commutes with $\iota_X$ and $\mathcal{L}_X$, for each $X\in\mathfrak{g}$. Choose $\tilde{X}^M\in\mathfrak{L}({\mathcal{F}})$ and $\tilde{X}^N\in\mathfrak{L}(\g)$ to be $f_t$-related foliate representatives of the transverse fields corresponding to the action of $X$ on ${\mathcal{F}}$ and $\g$, respectively. Then
\begin{align*}
f_t^*(\iota_{V_t}\iota_X\omega)_x(v_1,...,v_{k-2}) & =  (\iota_X\omega)_{f_t(x)}(V_t(x),\dif (f_t)_xv_1,\dots, \dif (f_t)_xv_{k-2})\\
 & =  \omega_{f_t(x)}(\tilde{X}^N_{f_t(x)},V_t(x),\dif (f_t)_xv_1,\dots, \dif (f_t)_xv_{k-2})\\
 & =  -\omega_{f_t(x)}(V_t(x),\dif (f_t)_x\tilde{X}^M_x,\dif (f_t)_xv_1,\dots, \dif (f_t)_xv_{k-2})\\
 & =  -f_t^*(\iota_{V_t}\omega)_x(\tilde{X}^M_x,v_1,\dots, v_{k-2})\\
 & =  -\iota_Xf_t^*(\iota_{V_t}\omega)_x(v_1,...,v_{k-2}).
\end{align*}
Hence the equality
\begin{equation}\label{eq: homotopy invariance 2} f_t^*(\iota_{V_t}\iota_X\omega)= -\iota_Xf_t^*(\iota_{V_t}\omega),\end{equation}
which integrated yields $\iota_XQ+Q\iota_X=0$. Now from Cartan's formula and equations \eqref{eq: homotopy invariance} and \eqref{eq: homotopy invariance 2},
\begin{align*}
\mathcal{L}_Xf_t^*(\iota_{V_t}\omega) & = \iota_X df_t^*(\iota_{V_t}\omega)+d\iota_Xf_t^*(\iota_{V_t}\omega)\\
 & = \iota_X\frac{d}{dt}f_t^*\omega-\iota_Xf_t^*(\iota_{V_t}d\omega)+d\iota_Xf_t^*(\iota_{V_t}\omega)\\
 & = \frac{d}{dt}f_t^*\iota_X\omega+f_t^*(\iota_{V_t}\iota_Xd\omega)-df_t^*(\iota_{V_t}\iota_X\omega)\\
 & = f^*_t(\iota_{V_t}d\iota_X\omega)+f_t^*(\iota_{V_t}\iota_Xd\omega)\\
 & = f^*_t(\iota_{V_t}\mathcal{L}_X\omega).
\end{align*}
Integrating this equation we get $\mathcal{L}_XQ-Q\mathcal{L}_X=0$.
\end{proof}

Two singular foliations $(M,{\mathcal{F}})$ and $(N,\g)$ with transverse $\mathfrak{g}$-actions are \textit{$\mathfrak{g}$-homotopy equivalent} when there exist foliate $\mathfrak{g}$-equivariant maps $f:M\to N$ and $g:N\to M$ such that $g\circ f\simeq \mathrm{id}_M$ and $f\circ g\simeq \mathrm{id}_N$ by foliate $\mathfrak{g}$-equivariant homotopies.

\begin{corollary}\label{corollary: invariance of ebc under homotopies and deformation retracts}
If ${\mathcal{F}}$ and $\g$ are $\mathfrak{g}$-homotopy equivalent singular foliations, then $H_{\mathfrak{g}}({\mathcal{F}})\cong H_{\mathfrak{g}}(\g)$ as $\sym(\mathfrak{g}^*)^{\mathfrak{g}}$-algebras. In particular, if $h:(M,{\mathcal{F}})\to (M,{\mathcal{F}})$ is a foliate $\mathfrak{g}$-equivariant deformation retraction onto a $\mathfrak{g}$-invariant submanifold $J\subset M$, then $H_{\mathfrak{g}}({\mathcal{F}})\cong H_{\mathfrak{g}}({\mathcal{F}}{\vert}_J)$.
\end{corollary}

\section{Molino's structural theory for Riemannian foliations}\label{section: (Singular) Riemannian foliations}

A singular foliation ${\mathcal{F}}$ is a \textit{singular Riemannian foliation} when there exists a Riemannian metric $\mathrm{g}$ on $(M,{\mathcal{F}})$ which is \textit{adapted} to ${\mathcal{F}}$, that is, such that every geodesic which is perpendicular to a leaf of ${\mathcal{F}}$ remains perpendicular to all leaves it intersects. We call such geodesics \textit{transnormal}. If the adapted metric $\mathrm{g}$ can be chosen to be complete, then we say that ${\mathcal{F}}$ is \textit{complete}. A main class of examples of singular Riemannian foliations is that of isometric homogeneous foliations.

\begin{example}[{\cite[Section 6.1]{molino}}]\label{example: homogeneous srf}
Let $(M,\mathrm{g})$ be a Riemannian manifold on which a Lie group $H$ acts by isometries. Then $\mathrm{g}$ is an adapted metric for ${\mathcal{F}}_H$, which is thus a singular Riemannian foliation.
\end{example}

When ${\mathcal{F}}$ is regular, the adapted metric $\mathrm{g}$ is also called \textit{bundle-like}. It induces a \textit{transverse metric} for ${\mathcal{F}}$, that is, a symmetric, basic $(2,0)$-tensor field $\mathrm{g}^T$ on $M$ satisfying $\mathrm{g}^T(v,v)>0$ whenever $v$ is not tangent to ${\mathcal{F}}$. This means $\mathrm{g}^T$ projects to a Riemannian metric on the quotient $S$ of a submersion $\pi_i:U\to S$ locally defining ${\mathcal{F}}$ by its fibers. A transverse field $X\in\mathfrak{l}({\mathcal{F}})$ is a \textit{transverse Killing vector field} when $\mathcal{L}_{\tilde{X}}\mathrm{g}^T=0$, for a representative $\tilde{X}\in \mathfrak{L}({\mathcal{F}})$. It is clear that such a field projects to a Killing vector field on $S$.

There is a rich structural theory which describes the leaf closures of a complete (regular) Riemannian foliation ${\mathcal{F}}$, known as Molino theory (we refer to \cite{molino} and \cite{alex2} for details on this topic). One of its results states that the collection $\overline{{\mathcal{F}}}=\{\overline{L}\subset M\ {\vert}\ L\in{\mathcal{F}}\}$ is a singular Riemannian foliation, and its leaves, the closures of leaves in ${\mathcal{F}}$, are determined by the orbits of a locally constant sheaf of Lie algebras of germs of transverse Killing vector fields, the Molino sheaf $\mathscr{C}_{\mathcal{F}}$. This means that, for each $x\in M$, there is a small neighborhood $U\ni x$ where 
$$T_y\overline{{\mathcal{F}}}(y)=T_y{\mathcal{F}}(y)\oplus\{X_y\ {\vert}\ X\in \mathscr{C}_{{\mathcal{F}}}(U)\}$$
for each $y\in U$. When $M$ is connected, the typical stalk $\mathfrak{g}_{\mathcal{F}}$ of the Molino sheaf is the \textit{structural algebra of ${\mathcal{F}}$}, an important invariant of the foliation.

A complete, regular Riemannian foliation with globally constant Molino sheaf is called a \textit{Killing foliation}. The terminology, from \cite{mozgawa}, is motivated by the case of a homogeneous foliation ${\mathcal{F}}_H$, given by the locally free, isometric action of a Lie group $H$ on a complete, connected Riemannian manifold $M$. In this case $\mathscr{C}_{{\mathcal{F}}}$ is the sheaf of germs of transverse Killing fields induced by the fundamental Killing vector fields of the action of $\overline{H}<\mathrm{Iso}(M)$. The structural algebra of a Killing foliation is Abelian \cite[Section 5.5]{molino}, and to emphasize this we will denote it by $\mathfrak{a}_{\mathcal{F}}$ (also omitting ${\mathcal{F}}$ if it is clear from the context). Since for a Killing foliation the sheaf $\mathscr{C}_{\mathcal{F}}$ is globally constant, we have $\mathfrak{a}\cong\mathscr{C}_{\mathcal{F}}(M)$. That is, $\mathfrak{a}$ acts transversely on ${\mathcal{F}}$, and its action describes the leaf closures. We denote this succinctly by $\mathfrak{a}{\mathcal{F}}=\overline{{\mathcal{F}}}$.

\subsection{Structural theory of singular Riemannian foliations}\label{section: molino theory for srf}

Now suppose ${\mathcal{F}}$ is a complete singular Riemannian foliation, and let $\Sigma^k$ be the union of all $k$-dimensional leaves of ${\mathcal{F}}$. We call $\Sigma^k$ the \textit{$k$-locus} of ${\mathcal{F}}$.  One proves that each connected component $\Sigma^k_\alpha$ of $\Sigma^k$ is an embedded submanifold of $M$, called a \textit{stratum} (see \cite[Section 6.2]{molino}). This gives a stratification
$$M=\bigsqcup_{k,\alpha} \Sigma^k_\alpha$$
of $M$. We will often omit $k$ in $\Sigma^k_\alpha$ when it is not needed. The restriction ${\mathcal{F}}_\alpha:={\mathcal{F}}{\vert}_{\Sigma_\alpha}$, endowed with $\mathrm{g}_\alpha:=\mathrm{g}{\vert}_{\Sigma_\alpha}$, is a regular Riemannian foliation, for each $\alpha$. Moreover, if $L\subset \Sigma_\alpha$ then $\overline{L}\subset\Sigma_\alpha$ (\cite[Lemma 6.4]{molino}), and each $\Sigma_\alpha$ is transversally totally geodesic, meaning that a geodesic which is normal to the leaves and tangent to $\Sigma_\alpha$ remains within it at least for some time (\cite[Proposition 6.3]{molino}). We will say that a transverse field $X\in\mathfrak{l}({\mathcal{F}})$ is a \textit{transverse Killing vector field} if its restriction to each stratum $\Sigma_\alpha$ is a transverse Killing vector field for $({\mathcal{F}}_\alpha,\mathrm{g}_\alpha^T)$.

It will be useful to define some more terminology. The $\dim({\mathcal{F}})$-locus of ${\mathcal{F}}$ is an open and dense submanifold of $M$, also called the \textit{regular locus} of ${\mathcal{F}}$ and denoted by $\Sigma_{\mathrm{reg}}$. Each connected component of $\Sigma_{\mathrm{reg}}$ is a \textit{regular stratum} of ${\mathcal{F}}$. All other strata are called \textit{singular} and have codimension at least $2$ (in particular, when $M$ is connected $\Sigma_{\mathrm{reg}}$ is connected, and hence a stratum). The union $\Sigma_{\mathrm{sing}}$ of all singular strata is the \textit{singular locus} of ${\mathcal{F}}$. The most singular strata, that is, those containing the leaves of least dimension, are called \textit{minimal}. Their union is the \textit{minimal locus} of ${\mathcal{F}}$, also denoted by $\Sigma_{\mathrm{min}}$. The \textit{closed locus} of ${\mathcal{F}}$ is the union of all closed leaves, denoted by $\Sigma_{\mathrm{cl}}$.

There is an analog of Molino's structural theorem for singular Riemannian foliations. Its main part is the fact that the collection $\overline{{\mathcal{F}}}$ of the closures of leaves of a complete singular Riemannian foliation ${\mathcal{F}}$ is again a singular Riemannian foliation, known as \emph{Molino's conjecture}. It remained open for more than three decades, and was proven to hold recently in \cite[Main Theorem]{alex4}.

The generalization of the Molino sheaf to the singular setting has a caveat. Although a restriction ${\mathcal{F}}_\alpha$ is not necessarily complete, one can still apply Molino's theory to it. This can be done by the approach to Molino theory via pseudogroups \cite[Appendix D]{molino}, since the holonomy pseudogroup of ${\mathcal{F}}_\alpha$ is complete. In particular, for the case of ${\mathcal{F}}_{\mathrm{reg}}$, it is possible to prove that the corresponding Molino sheaf $\mathscr{C}_{\mathrm{reg}}$ extends continuously to a locally constant sheaf $\mathscr{C}_{{\mathcal{F}}}$ on $M$, called the \textit{Molino sheaf of ${\mathcal{F}}$} (see \cite[Lemma 6.5]{molino}). And this is the caveat: a section of $\mathscr{C}_{\mathrm{reg}}$ admits a continuous extension through the singular locus, but we do not know whether in general this extension is smooth. In other words, a section of $\mathscr{C}_{{\mathcal{F}}}$ is a ``continuous transverse Killing vector field'' $X$: it restricts to a transverse Killing vector field on each stratum, but it is represented by a local vector field $\tilde{X}$ on $M$ which is only continuous.

The question of whether in general $\mathscr{C}_{\mathrm{reg}}$ extends smoothly to a sheaf of Lie algebras of germs of genuine transverse Killing vector fields is known as the \emph{strong Molino conjecture} (since it implies the usual one). Molino states this conjecture in \cite[p. 215]{molino}. We will see that it holds for an important class of foliations in Theorem \ref{proposition: strong molino conjecture for orbit-like}. In general, we will say that ${\mathcal{F}}$ is a singular Riemannian foliation \textit{with smooth $\mathscr{C}_{{\mathcal{F}}}$} when it satisfies the strong Molino conjecture. Another relevant class of Riemannian foliations with this property is the following.

\begin{example}[Molino sheaf of a homogeneous singular Riemannian foliation]\label{exe: molino sheaf of homogeneous singular foliation}
Let $(M,\mathrm{g})$ be a complete Riemannian manifold with an isometric homogeneous foliation ${\mathcal{F}}$, given by the connected components of the orbits of $H<\mathrm{Iso}(M)$. The fundamental vector fields of the action of the closure $\overline{H}<\mathrm{Iso}(M)$ are $\mathrm{g}$-Killing and foliate. The Molino sheaf $\mathscr{C}_{{\mathcal{F}}}$ is the sheaf of Lie algebras of germs of transverse Killing vector fields induced by them. Hence $\overline{{\mathcal{F}}}$ is also homogeneous, given by the connected components of the orbits of $\overline{H}$.
\end{example}

The Molino sheaf $\mathscr{C}_\alpha$ of each ${\mathcal{F}}_\alpha$ is the quotient of the restriction of $\mathscr{C}_{{\mathcal{F}}}$ to $\Sigma_\alpha$ by the kernel of the restriction map on sections, that is, by the subsheaf consisting of those sections whose restriction to $\Sigma_\alpha$ vanish \cite[Proposition 6.8]{molino}. It is clear, hence, that the orbits of $\mathscr{C}_{{\mathcal{F}}}$ describe the closures of the leaves of ${\mathcal{F}}$. In analogy with the regular case, when $M$ is connected we say that the typical stalk $\mathfrak{g}_{\mathcal{F}}$ of $\mathscr{C}_{{\mathcal{F}}}$ is the \textit{structural Lie algebra} of ${\mathcal{F}}$. It coincides with the structural algebra of ${\mathcal{F}}_{\mathrm{reg}}$. From the relation between $\mathscr{C}_\alpha$ and $\mathscr{C}_{{\mathcal{F}}}$, it follows that the structural algebra $\mathfrak{g}_\alpha$ of ${\mathcal{F}}_\alpha$ is a quotient of $\mathfrak{g}_{\mathcal{F}}$.

\begin{example}[Structural algebra of a homogeneous singular Riemannian foliation]\label{exe: structural algebra of homogeneous singular foliation}
In the context of Example \ref{exe: molino sheaf of homogeneous singular foliation}, the structural algebra $\mathfrak{g}_{\mathcal{F}}$ is isomorphic to $\mathrm{Lie}(\overline{H})/\mathrm{Lie}({H})$.
\end{example}

\subsection{Homotheties}

Let $(M,{\mathcal{F}})$ be a complete singular Riemannian foliation and let $J\subset M$ be a saturated submanifold contained in some stratum $\Sigma^k$. Let $P\subset J$ be a connected open subset that admits a tubular neighborhood $\mathrm{Tub}_r(P)$, for some $r>0$. That is, $\mathrm{Tub}_r(P)$ is the diffeomorphic image of $\{v\in\nu P\ {\vert}\ \|v\|<r\}$ under the normal exponential map $\exp^\perp:\nu J\to M$. This is the case, for instance, when $P$ is relatively compact in $J$. We denote by $\rho_P:\mathrm{Tub}_r(P)\to P$ the orthogonal projection. By shrinking $r$ if necessary, we can assume further that for any $y\in \mathrm{Tub}_r(P)$ the leaf ${\mathcal{F}}(y)$ is transverse to $\rho_P^{-1}(\rho_P(y))$. Then each connected component of ${\mathcal{F}}(y)\cap \mathrm{Tub}_r(P)$ is a \textit{plaque} of ${\mathcal{F}}$, and the neighborhood $\mathrm{Tub}_r(P)$ is called a \textit{distinguished tubular neighborhood} for $P$.

Any closed leaf $L\in {\mathcal{F}}$ has a distinguished tubular neighborhood, even if it is not compact. In fact, it is shown in \cite[Proposition 16]{mendes} that for a fixed $x\in L$, if $r>0$ is such that $\exp_x(tv){\vert}_{t\in[0,1]}$ is the unique minimizing geodesic between $L$ and $\exp_x(v)$, for all $v\in\nu_x L$ with $\|v\|<r$, then $\mathrm{Tub}_r(L)$ is a tubular neighborhood for $L$. The same holds for any leaf closure $\overline{L}\in\overline{{\mathcal{F}}}$, since $\overline{{\mathcal{F}}}$ is again a singular Riemannian foliation.

For all $\lambda\in(0,\infty)$ such that $\mathrm{Tub}_{\lambda r}(P)$ is a distinguished tubular neighborhood of $P$ (let us call such a $\lambda$ \textit{admissible}), we define the \textit{homothetic transformation}
$$h_\lambda:\mathrm{Tub}_{r}(P)\ni\exp^\perp(v)\longmapsto \exp^\perp(\lambda v)\in \mathrm{Tub}_{\lambda r}(P)$$
around $P$. We also say that $P$ is the \textit{homothety center} of $h_\lambda$. Notice that each $h_\lambda$ is a diffeomorphism, and $h_\lambda\circ h_\mu=h_{\lambda\mu}$, when both sides of the equation make sense. We extend the definition to $\lambda=0$, getting $h_0=\rho_P$. One checks that in fact $h_\lambda$ converges to $\rho_P$ in the compact-open topology, as $\lambda\to 0$ (see \cite[Section 2.2]{inagaki}). These maps are of particular interest in the theory of singular Riemannian foliations because of Molino's homothetic lemma, which asserts that they are foliate with respect to the restrictions of ${\mathcal{F}}$ to $\mathrm{Tub}_{r}(P)$ and $\mathrm{Tub}_{\lambda r}(P)$ (see \cite[Lemma 6.2]{molino} for the case $\lambda>0$; for $\lambda=0$ this follows easily from the properties of ${\mathcal{F}}$). Special cases of interest are $P=\overline{L}$ and $P=\Sigma_\alpha\subset\Sigma_{\mathrm{min}}$.

Let $\mathrm{Tub}_r(P)$ be a distinguished tubular neighborhood and $U\subset\mathrm{Tub}_r(P)$ a saturated set which is preserved by $h_\lambda$, for $\lambda\in[0,1]$. Let us say $U$ is \textit{$P$-star-shaped}. Since homothetic transformations on $\mathrm{Tub}_r(P)$ restrict to the identity on $P$ and are defined for $\lambda\in[0,1]$, they provide a \textit{homothetic retraction}
$$h_P:U\times[0,1]\longrightarrow U$$
of $U$ onto $P$ by $h_P(x,t)=h_{1-t}(x)$. These maps will be very useful to us, since they are foliate strong deformation retractions.

\subsection{Orbit-like foliations}\label{section: orbit-like foliations}

In this section we prove Theorem \ref{theoremA} (as Theorem \ref{proposition: strong molino conjecture for orbit-like} below). We start by introducing the class of foliations for which it is stated. Let ${\mathcal{F}}$ be a complete singular Riemannian foliation and consider a distinguished tubular neighborhood $\mathrm{Tub}_r(P)$ for $P\subset L\in{\mathcal{F}}$. A fiber $S_x=\rho_P^{-1}(x)$ is a \textit{slice} for ${\mathcal{F}}$ at $x\in P$, and ${\mathcal{F}}{\vert}_{S_x}$ is called the \textit{slice foliation} at $x$. Its pullback by $\exp^\perp_x$ is a singular Riemannian foliation of $B_r(0)\subset T_xS_x=\nu_xL$ which, by Molino's homothetic lemma, can be extended via homotheties to a singular Riemannian foliation ${\mathcal{F}}_x$ on the whole of $\nu_xL$. This extension is the so-called \textit{infinitesimal foliation at $x$}.

Some relevant subclasses of singular Riemannian foliations are defined in terms of infinitesimal foliations. For instance, ${\mathcal{F}}$ is called
\begin{enumerate}
\item \textit{infinitesimally closed} when ${\mathcal{F}}_x$ is closed for each $x\in M$,
\item \textit{infinitesimally homogeneous} when ${\mathcal{F}}_x$ is homogeneous for each $x\in M$, and
\item \textit{orbit-like} when it is infinitesimally closed and infinitesimally homogeneous.
\end{enumerate} 
By Molino's theory, the closure $\overline{{\mathcal{F}}}$ of a complete regular Riemannian foliation ${\mathcal{F}}$ is orbit-like, which led him to study this class of foliations (see, e.g., \cite{molino2}). Other examples of orbit-like foliations are given by the so-called \textit{holonomy foliations}, whose leaves are the holonomy tubes with respect to a metric connection on an Euclidean vector bundle over a Riemannian manifold $L$ (see, e.g., \cite[Example 2.7]{alex4}). Orbit-like and infinitesimally closed foliations were also prominent in the proof of Molino's conjecture in \cite{alex4} (also \cite{alex9}). In fact, let us now see that the machinery developed in \cite{alex9} adapt directly to conclude that the strong Molino conjecture holds for orbit-like foliations.

Let $\Sigma$ be the stratum containing $x$ and let $V=\mathrm{Tub}_r(P)\subset U$ be a distinguished tubular neighborhood around an open homothety center $P\subset\Sigma$ containing $x$. Take a slice $S\subset P$ through $x$ for ${\mathcal{F}}{\vert}_P$. We are interested in the so called \textit{local reduction at $x$}
$$N=\exp(\{v\in (\nu P){\vert}_S\ {\vert}\ \|v\|<r\}),$$
which can be thought as a generalized slice (see \cite[Section 2.8]{alex9}). As in the case of slices, by shrinking $V$ if necessary we can assume that there is a submersion $p:V\to N$, whose fibers are contained in the leaves of ${\mathcal{F}}$. The connected components of the intersections of the leaves of ${\mathcal{F}}$ with $N$ define a singular foliation ${\mathcal{F}}_N$, which is in general more sensible to the dynamics of ${\mathcal{F}}$ then the infinitesimal foliation ${\mathcal{F}}_x$. We now cite some results of \cite{alex9} concerning $(N,{\mathcal{F}}_N)$ that will be useful.

First, in \cite[Proposition 2.20]{alex9} the authors construct a Riemannian metric $\mathrm{g}_{N}$ on $N$ that is adapted to ${\mathcal{F}}_{N}$, hence turning it into a singular Riemannian foliation. Furthermore, $\mathrm{g}_{N}$ preserves the transverse metric of ${\mathcal{F}}$, meaning that the distance between leaves of ${\mathcal{F}}_{N}$ with respect to $\mathrm{g}_{N}$ coincides with the distance between the corresponding plaques of ${\mathcal{F}}$ defining them. We will also use \cite[Corollary 2.25]{alex9}, which establishes that, when ${\mathcal{F}}$ is orbit-like, the foliation ${\mathcal{F}}_N$ is a homogeneous singular foliation, given by the orbits of a compact Lie group (notice that the assumption on \cite[p. 13]{alex9} that $(M, \mathcal{F})$ is \textit{locally closed}, under which \cite[Corollary 2.25]{alex9} falls, is automatically satisfied when $\mathcal{F}$ is orbit-like).

The idea of the proof of Theorem \ref{theoremA} is to smoothly lift the local isometric flow $\varphi$ on $M/\f$ induced by a section of $\mathscr{C}_{\mathcal{F}}$. For that, as noted in \cite{alex9}, one can apply G.~Schwarz's fundamental result on the isotopy lifting conjecture: a smooth flow on the orbit space $M/G$ of a proper action $G\times M\to M$ is the projection of a smooth $G$-equivariant flow on $M$ \cite[Corollary 2.4]{schwarz}. Finally, to guarantee that $\varphi$ is smooth in order to apply Schwarz's theorem, we will use \cite[Theorem 3.1]{alex9}: if $(M,\f)$ is a singular Riemannian foliation whose leaves are spanned by a proper smooth action $G\times M\to M$ and $\varphi$ is a continuous local flow of isometries on the orbit space, then $\varphi$ is smooth.

\begin{theorem}[Strong Molino conjecture for orbit-like foliations]\label{proposition: strong molino conjecture for orbit-like}
If ${\mathcal{F}}$ is orbit-like, then $\mathscr{C}_{\mathcal{F}}$ is smooth.
\end{theorem}

\begin{proof}
Let $X\in \mathscr{C}_{\mathcal{F}}(U)$ be a local section and let $x\in U$. We must prove that $X$ admits a smooth representative $\tilde{X}$ in some neighborhood $V\ni x$, which we will take to be a distinguished tubular neighborhood defining a local reduction $N$ at $x$. Let $\mathrm{g}_{N}$ be the aforementioned ${\mathcal{F}}_N$-adapted metric on $N$ \cite[Proposition 2.20]{alex9}. Since $\mathrm{g}_{N}$ preserves the transverse metric of ${\mathcal{F}}$, it follows that the restriction $X{\vert}_N$ defines a continuous isometric flow $\varphi$ on $N/{\mathcal{F}}_{N}$, as it restricts to a transverse Killing vector field on the dense open subset $\Sigma_{\mathrm{reg}}\cap N$.

Now, as we saw above, ${\mathcal{F}}_N$ is homogeneous, given by the orbits of a compact Lie group \cite[Corollary 2.25]{alex9}. Hence, by \cite[Theorem 3.1]{alex9}, the local flow $\varphi$ is smooth. By Schwarz's theorem, it hence follows that $\varphi$ lifts to an equivariant smooth flow $\tilde{\varphi}$ on $N$. The infinitesimal generator $\tilde{X}_{N}$ of $\tilde{\varphi}$ is then a smooth representative for $X{\vert}_N$. Having that, we can now just choose any smooth field $\tilde{X}$ on $V$ that is $p$-related to $\tilde{X}_{N}$, where $p:V\to N$ is the submersion of $V$ onto $N$. Since $X$ is ${\mathcal{F}}$-invariant and the fibers of $p$ are contained in the leaves of ${\mathcal{F}}$, it follows that $\tilde{X}$ is a local representative of $X$.
\end{proof}

\section{Good covers and the dimension of basic cohomology}\label{section: good covers}

We say that a locally finite cover $\mathcal{U}=\{U_i\}$ of $M$ by saturated open sets is a \textit{good cover} for $(M,{\mathcal{F}})$ when any non-empty intersection $U$ of elements of the cover admits a foliate deformation retraction onto the closure of some leaf $L_U\subset U$. We will prove here that singular Riemannian foliations of compact manifolds admit finite good covers, whose retractions are furthermore homothetic (we will actually obtain this more generally for transversely compact foliations, see Section \ref{subsection: existence good covers}). This result were previously stated in \cite{wolak} for a singular Riemannian foliation of a compact manifold $M$, but we will provide an alternative proof here, because the original one relies on \cite[Lemma 2]{wolak}, which does not hold as stated. It claims that there exists $r>0$ which works as the radius of a tubular neighborhood $\mathrm{Tub}_r(\overline{L})$ for any $L\in {\mathcal{F}}$. This can fail because regular leaves near a singular leaf $L$ must lie inside distance tubes $\partial \mathrm{Tub}_r(L)$, and so their focal radii tend to zero as they get closer to $L$. The claim also fails for a regular foliation if it has at least one leaf $L$ with non-trivial holonomy: the recurrence of generic nearby leaves makes their normal injectivity radii also tend to zero as they get closer to $L$.

\subsection{Holonomy types and transverse convexity}

Guided by the discussion above, we will define a finer stratification of $M$ that discriminate leaves based not only on their dimension but also on their holonomy. Let ${\mathcal{F}}$ be a complete closed singular Riemannian foliation of $M$. The restriction ${\mathcal{F}}_\alpha$ to each stratum $\Sigma_\alpha$ is a closed regular Riemannian foliation. Each $\Sigma_\alpha/{\mathcal{F}}_\alpha$ is then a Riemannian orbifold, whose local groups are the holonomy groups of the leaves, and the projection $\pi_\alpha:\Sigma_\alpha\to \Sigma_\alpha/{\mathcal{F}}_\alpha$ is an orbifold map (see \cite[Proposition 3.7]{molino}; it is stated for ${\mathcal{F}}$ with compact leaves, but notice the proof there works when ${\mathcal{F}}$ is only closed and complete). Let $\Sigma_\alpha(\Gamma)$ be the inverse image under $\pi_\alpha$ of the locus of points in $\Sigma_\alpha/{\mathcal{F}}_\alpha$ with local group $\Gamma$ (defined up to action isomorphism). It is a totally geodesic manifold (see \cite[Lemma 4.5.3 and Theorem 4.5.4]{choi}), thus $\Sigma_\alpha(\Gamma)$ is a transversely totally geodesic submanifold. We say that a leaf $L\subset\Sigma_\alpha(\Gamma)$ has \textit{holonomy type $\Gamma$}. This furnishes the stratification
$$M=\bigsqcup\Sigma^{k}_{\alpha}(\Gamma)$$
by holonomy type. For simplicity, we will denote a generic holonomy type stratum by $\mathcal{S}$, and the one containing $L\in{\mathcal{F}}$ by $\mathcal{S}_L$.

We will need a type of uniformly normal saturated neighborhood. As already noted, one cannot expect to find $r>0$ that works as the radius of a tubular neighborhood $\mathrm{Tub}_r(\overline{L})$ for any leaf in some saturated neighborhood. Instead, we will use the following weaker notion: for $\mathcal{S}$ a holonomy type stratum of ${\mathcal{F}}$, we say that a saturated open set $W\subset M$ is \textit{uniformly $\mathcal{S}$-normal} if there exists $r>0$ such that $W\subset \mathrm{Tub}_r(L)$ for every leaf $L\subset W\cap\mathcal{S}$.

\begin{lemma}\label{lemma: uniformly normal}
Let ${\mathcal{F}}$ be a closed, complete singular Riemannian foliation. Given a leaf $L\in{\mathcal{F}}$ and any saturated neighborhood $U\supset L$, there exists a uniformly $\mathcal{S}_L$-normal open set $W\subset U$ containing $L$.
\end{lemma}

\begin{proof}
Since ${\mathcal{F}}{\vert}_{\mathcal{S}_L}$ is a simple Riemannian foliation, hence a locally trivial fiber bundle, we can find a local reduction $N\subset U$ for ${\mathcal{F}}$ such that $N\cap\mathcal{S}_L$ is a trivializing slice for ${\mathcal{F}}{\vert}_{\mathcal{S}_L}$. Recall there is a metric $\mathrm{g}_N$ on $N$ that preserves the transverse metric of ${\mathcal{F}}$ (see \cite[Proposition 2.20]{alex9}). Then, if $B_r(x)\subset N$ is a geodesic ball at $x\in N\cap\mathcal{S}_L$, we have that $\mathrm{Tub}_r({\mathcal{F}}(x))$ is a tubular neighborhood for ${\mathcal{F}}(x)$. Therefore it is sufficient to take $W$ as the saturation of a uniformly normal neighborhood of (the singleton) $L\cap N$ in $N$.
\end{proof}

We say that $W\supset L$, given by Lemma \ref{lemma: uniformly normal}, is a \textit{uniformly $\mathcal{S}_L$-normal neighborhood} of $L$. We will also need tubular neighborhoods with a type of transverse convexity. A saturated open set $U\subset M$ is \textit{strongly transversely convex} when for any two leaf closures in $U$ there is a transversely unique minimal geodesic segment $\gamma$ that lies completely in $U$ connecting them. By \textit{transversely unique} we mean that if $\sigma$ is another minimal geodesic between those leaf closures, with $\gamma(0)=\sigma(0)$ and $\gamma(1)=\sigma(1)$, then $\sigma=\gamma$.

\begin{lemma}\label{lemma: strongly transversely geodesic}
Let ${\mathcal{F}}$ be a closed, complete singular Riemannian foliation. For each $L\in{\mathcal{F}}$ there exists $c>0$ such that $\mathrm{Tub}_r(L)$ is strongly transversely geodesic, for any $0<r<c$.
\end{lemma}

\begin{proof}
The proof of \cite[Lemma A.3]{toben} for the homogeneous case adapts directly to the above setting since it does not use the homogeneity.
\end{proof}

\subsection{Existence of good covers}\label{subsection: existence good covers}

We say that a singular foliation ${\mathcal{F}}$  is \textit{transversely compact} when $M/\f$ is compact. Of course, if $M$ is compact then any singular foliation of $M$ is transversely compact. Recall that, if $\f$ is a complete singular Riemannian foliation, then $\overline{{\mathcal{F}}}$ is again a singular Riemannian foliation (\cite[Main Theorem]{alex4}). In this case ${\mathcal{F}}$ is transversely compact if, and only if, $\overline{{\mathcal{F}}}$ is transversely compact, because ${\mathcal{F}}$ and $\overline{{\mathcal{F}}}$ have the same saturated open sets in $M$. In fact, it is clear that $\overline{{\mathcal{F}}}$-saturated sets are ${\mathcal{F}}$-saturated. For the converse, let $U$ be open and $\mathcal{F}$-saturated, containing $L\in\mathcal{F}$, and suppose $\overline{L}\not\subset U$. Then there is some $x\in\overline{L}$ with $x\notin U$, hence ${\mathcal{F}}(x)\cap U=\emptyset$, thus $\overline{\mathcal{F}(x)}\cap U=\emptyset$, which is a contradiction since $\overline{\mathcal{F}(x)}=\overline{L}$.

\begin{proposition}\label{proposition: good cover}
A transversely compact, complete singular Riemannian foliation ${\mathcal{F}}$ admits a finite good cover $\mathcal{U}$ with the additional property that each foliate deformation retraction of a non-empty intersection of its elements onto a leaf closure in it is a homothetic retraction.
\end{proposition}

\begin{proof}
We begin by noticing that we can assume ${\mathcal{F}}$ is closed: a cover for $\overline{{\mathcal{F}}}$ with the desired properties will work for ${\mathcal{F}}$ as well, by what we just saw above.

Let $L\in{\mathcal{F}}$. By combining Lemmas \ref{lemma: uniformly normal} and \ref{lemma: strongly transversely geodesic}, we can choose $c_L>0$ such that $\mathrm{Tub}_{r}(L)$ is strongly transversely geodesic and uniformly $\mathcal{S}_L$-normal, for all $0<r<c_L$. Notice also that, since $M/{\mathcal{F}}$ is compact, there are finitely many holonomy type strata $\mathcal{S}(i)=\Sigma^{k(i)}_{\alpha(i)}(\Gamma(i))$, and we can assume they are ordered so that $i<j$ implies that $k(i)\leq k(j)$, and if moreover $\alpha(i)=\alpha(j)$, that $\Gamma(j)$ is conjugate to a subgroup of $\Gamma(i)$.

We now cover $M$ inductively as follows. Let $\mathcal{S}(1),\dots,\mathcal{S}(\ell_1)$ be the closed holonomy type strata. They are all disjoint, so we can choose tubular neighborhoods $\mathrm{Tub}_{r_L}(L)$ so that
$$\bigsqcup_{i=1}^{\ell_1}\mathcal{S}(i)\subset \bigsqcup_{i=1}^{\ell_1} \left(\bigcup_{L\in \mathcal{S}(i)} \mathrm{Tub}_{r_L}(L)\right)=:V(1).$$
Notice that in fact we can suppose the unions $\bigcup_{L\in \mathcal{S}(i)} U_L$ are pairwise disjoint by choosing $r_L<c_L$ appropriately. For the generic $k$-th step, we select the strata $\mathcal{S}(\ell_{k-1}+1),\dots,\mathcal{S}(\ell_k)$ for which $\mathcal{S}'(i):=\mathcal{S}(i)\setminus \bigcup_{j=1}^{k-1}V(j)$ is closed and non-empty. Then similarly, we take the cover
$$\bigsqcup_{i=\ell_{k-1}+1}^{\ell_k}\mathcal{S}'(i)\subset \bigsqcup_{i=\ell_{k-1}+1}^{\ell_k} \left(\bigcup_{L\in \mathcal{S}'(i)} \mathrm{Tub}_{r_L}(L)\right)=:V(k),$$
choosing $r_L$ so that furthermore $V(k)\cap\bigsqcup_{i=1}^{\ell_{k-1}}\mathcal{S}(i)=\emptyset$.

The process will end since there are finitely many holonomy type strata, furnishing us a cover $\{U_L:=\mathrm{Tub}_{r_L}(L)\}_{L\in{\mathcal{F}}}$. Since ${\mathcal{F}}$ is transversely compact and this cover is by saturated open sets, we can choose a finite subcover $\mathcal{U}$. Now notice that if $U=U_1\cap\dots\cap U_l$ with $U_i\in\mathcal{U}$, then by construction ${\mathcal{F}}{\vert}_U$ has a unique holonomy type stratum which is closed. In fact, it is $\mathcal{S}(i_U)\cap U$, where $i_U=\min\{i\ {\vert}\ \mathcal{S}(i)\cap U\neq \emptyset\}$. Then for any leaf $L\in \mathcal{S}(i_U)\cap U$, the set $U$ is contained in a tubular neighborhood $\mathrm{Tub}_r(L)$, since at least one of the sets $U_1,\dots, U_l$ is uniformly $\mathcal{S}_L$-normal. Moreover, $U$ is strongly transversely geodesic since each $U_1,\dots, U_l$ has this property. Therefore $U$ retracts by homothety to $L$. 
\end{proof}

\subsection{Basic cohomology is finite dimensional}

We can now apply the classical Mayer--Vietoris argument to conclude that $H({\mathcal{F}})$ is finite dimensional.

\begin{theorem}\label{theorem: basic cohomology is finite dimensional}
Let ${\mathcal{F}}$ be a transversely compact, complete singular Riemannian foliation. Then $\dim H({\mathcal{F}})<\infty$.
\end{theorem}

\begin{proof}
Let us prove this, more generally, for any $(M,\f)$ which admits a finite good cover and whose leaf closures are complete and transversely compact. We proceed by induction on the number $n$ of elements in a finite good cover $\mathcal{U}$. If $n=1$ then $M$ itself admits a foliate deformation retraction to a leaf closure $J=\overline{L}$. By Corollary \ref{corollary: invariance of ebc under homotopies and deformation retracts} (taking $\mathfrak{g}=0$) we have $H({\mathcal{F}})\cong H({\mathcal{F}}{\vert}_J)$. Since ${\mathcal{F}}{\vert}_J$ is a regular complete Riemannian foliation of the transversely compact manifold $J$, its basic cohomology is finite dimensional (see \cite[Proposition 3.11]{goertsches}, also \cite[Theorem 0]{kacimi}).

For the induction step, let $\mathcal{U}=\{U_1,\dots, U_n\}$ be a good cover. Defining $U=\bigcup_{i=2}^n U_i$, we easily see that both ${\mathcal{F}}{\vert}_U$ and ${\mathcal{F}}{\vert}_{U_1\cap U}$ are Riemannian foliations admitting good covers with $n-1$ elements. Thus, by the induction hypothesis, both extremes in the exact sequence
$$H^{k-1}({\mathcal{F}}{\vert}_{U_1\cap U})\longrightarrow H^{k}({\mathcal{F}}) \longrightarrow H^k({\mathcal{F}}{\vert}_{U_1})\oplus H^k({\mathcal{F}}{\vert}_{U}),$$
are finite dimensional, for each $k$. Hence $\dim H({\mathcal{F}})<\infty$.
\end{proof}

\section{The natural transverse action on a singular Killing foliation}\label{section: transverse action of structural algebra on a skf}

In analogy with the regular case, we define a \textit{singular Killing foliation} as a complete singular Riemannian foliation ${\mathcal{F}}$ whose Molino sheaf $\mathscr{C}_{\mathcal{F}}$ is a globally constant sheaf of Lie algebras of germs of transverse Killing vector fields. In other words, ${\mathcal{F}}$ is a complete singular Riemannian foliation with a globally constant, smooth $\mathscr{C}_{\mathcal{F}}$. An important class of Killing foliations is that of isometric homogeneous foliations.

\begin{example}[Homogeneous Riemannian foliations are Killing]\label{example: homogeneous Riemannian foliations are Killing}
Let $(M,{\mathcal{F}})$ be a complete homogeneous Riemannian foliation, given by the orbits of an isometric action. Then ${\mathcal{F}}$ is Killing, as it is clear from Example \ref{exe: molino sheaf of homogeneous singular foliation} that $\mathscr{C}_{\mathcal{F}}$ is a globally constant sheaf of Lie algebras of germs of transverse Killing vector fields.
\end{example}

Another important class of examples is the following.

\begin{example}[Orbit-like Riemannian foliations on simply connected manifolds are Killing]\label{orbit-like on 1connected are killing}
As we saw in Theorem \ref{proposition: strong molino conjecture for orbit-like}, if ${\mathcal{F}}$ is an orbit-like foliation then $\mathscr{C}_{\mathcal{F}}$ is smooth. Hence ${\mathcal{F}}$ is Killing when $M$ is simply connected, since in this case $\mathscr{C}_{\mathcal{F}}$ has trivial monodromy and is thus globally constant.
\end{example}

Let ${\mathcal{F}}$ be a singular Killing foliation. As we saw in Section \ref{section: (Singular) Riemannian foliations}, the Molino sheaf $\mathscr{C}_\alpha$ of the restriction ${\mathcal{F}}_\alpha$ of ${\mathcal{F}}$ to a stratum $\Sigma_\alpha$ is a quotient of $\mathscr{C}_{{\mathcal{F}}}$, hence constant. We already mentioned in Section \ref{section: molino theory for srf} that ${\mathcal{F}}_\alpha$ is not complete in general, but its holonomy pseudogroup is, so we can say that each ${\mathcal{F}}_\alpha$ is a regular Killing foliation in this generalized sense (i.e., with completeness of the metric changed to completeness of the pseudogroup). In particular, the structural algebra $\mathfrak{a}$ of ${\mathcal{F}}$ is Abelian (notice that $\mathfrak{a}$ is well defined even if $M$ is not connected, since $\mathscr{C}_{\mathcal{F}}$ is a constant sheaf). It acts transversely on ${\mathcal{F}}$ via the isomorphism $\mathfrak{a}\cong \mathscr{C}_{\mathcal{F}}(M)$, and $\mathfrak{a}{\mathcal{F}}=\overline{{\mathcal{F}}}$. From what we saw in Section \ref{section: equivariant basic cohomology}, with this natural transverse $\mathfrak{a}$-action, $\Omega({\mathcal{F}})$ becomes an $\mathfrak{a}^\star$-algebra and we can consider the equivariant cohomology $H_{\mathfrak{a}}({\mathcal{F}})$. It is then expected that $H_{\mathfrak{a}}({\mathcal{F}})$ carries information on the closed leaves of ${\mathcal{F}}$, since
$$M^{\mathfrak{a}}=\Sigma_{\mathrm{cl}}.$$
Before we investigate that, let us establish some basic tools, the first one being the fact that the $\mathfrak{a}$-action behaves well under homothetic transformations.

\subsection{Homothetic invariance of the Molino sheaf}\label{section: homothetic invariance of Molino sheaf}

Let $\mathrm{Tub}_{r}(P)$ be a distinguished tubular neighborhood and $X$ be a smooth vector field on $\mathrm{Tub}_{r}(P)$ which is tangent to $P$. The \textit{linearization of $X$ around $P$} is the vector field $X^\ell$ given by
$$X^\ell=\lim_{\lambda\to 0} \dif h_\lambda^{-1}(X\circ h_\lambda).$$
By choosing adequate coordinates one verifies that $X^\ell$ is in fact well defined, smooth, $h_\lambda$-invariant for all admissible $\lambda$ (including $\lambda=0$) and moreover $X^\ell{\vert}_P=X{\vert}_P$ (see \cite[Proposition 13]{mendes}). Furthermore, $X^\ell$ is foliate when $X$ is foliate (see \cite[Proposition 2.14]{inagaki}). The next lemma will help us prove that sections of $\mathscr{C}_{{\mathcal{F}}}$ are invariant under certain homothetic transformations.

\begin{lemma}\label{lemma: invariance of molino sheaf under homothetic transformations}
Let $(M,{\mathcal{F}})$ be a complete singular Riemannian foliation with smooth Molino sheaf, $J\subset M$ an $\overline{{\mathcal{F}}}$-saturated homothety center, and $h_\lambda:\mathrm{Tub}_{r}(J)\to \mathrm{Tub}_{\lambda r}(J)$ a homothetic transformation, with $\lambda>0$. If $\tilde{X}$ is a foliate representative of a section $X\in\mathscr{C}_{\mathcal{F}}(\mathrm{Tub}_r(P))$, for $P\subset J$ an open subset, then $\dif (h_\lambda)_x\tilde{X}_x-\tilde{X}_{h_\lambda(x)}\in T_{h_\lambda(x)}{\mathcal{F}}$, for all $x\in \mathrm{Tub}_r(P)$.
\end{lemma}

\begin{proof}
It suffices to show this for $\lambda\in(0,1]$. In fact, assuming this holds, the result for an admissible $\lambda>1$ follows by the chain rule (as $h_\lambda\circ h_{1/\lambda}=\mathrm{id}_{\mathrm{Tub}_r(J)}$), and the fact that homothetic transformations are foliate.

Since $h_t\circ h_s=h_{ts}$, we have that $h_{e^t}(x)$ defines a (local) flow on $\mathrm{Tub}_{r}(J)$. Consider the infinitesimal generator $\tilde{R}$ of this flow, that is
$$\tilde{R}(x)=\left.\od{}{t}h_{e^t}(x)\right\vert_{t=0}.$$
It is the foliate radial vector field around $J$, and in particular it is tangent to each stratum of ${\mathcal{F}}$ that it intersects. For a fixed $x\in \mathrm{Tub}_r(P)$, take $r'<r$ with $x\in\mathrm{Tub}_{r'}(J)$ and extend $\tilde{R}$ to a global foliate field by multiplying it with a basic bump function for $\mathrm{Tub}_{r'}(J)$ with support in $\overline{\mathrm{Tub}_{r}(J)}$ and defining $\tilde{R}=0$ outside that set. We will denote the flow of $\tilde{R}$ by $\theta:\mathbb{R}\times M\to M$.

Now since $X$ is a section of $\mathscr{C}_{\mathcal{F}}$, its restriction $X{\vert}_\Sigma$ to the stratum $\Sigma\ni x$ commutes with each global transverse vector field in $\mathfrak{l}({\mathcal{F}}{\vert}_\Sigma)$ (see \cite[p. 160]{molino}), in particular with the transverse field induced by the restriction $\tilde{R}{\vert}_\Sigma$. This means $[\tilde{X}{\vert}_\Sigma,\tilde{R}{\vert}_\Sigma]$ is tangent to ${\mathcal{F}}$ on $U=\Sigma\cap\mathrm{Tub}_{r}(P)$. Define $\tilde{Z}(t)=d(\theta_{-t})_{\theta_t(x)}(\tilde{X}_{\theta_t(x)})$. Then
$$\left.\od{}{t}\tilde{Z}(t)\right\vert_{t=s}=\dif(\theta_{-s})[\tilde{X}{\vert}_\Sigma,\tilde{R}{\vert}_\Sigma]_{\theta_s(x)}\in T_x{\mathcal{F}},$$
since $\theta_{-s}$ is foliate. Moreover, $\tilde{Z}(0)=\tilde{X}_x$, therefore
$$\tilde{Z}(t)-\tilde{X}_x=\int_0^t \dif(\theta_{-s})[\tilde{X}{\vert}_\Sigma,\tilde{R}{\vert}_\Sigma]_{\theta_s(x)}\;ds\in T_x{\mathcal{F}}$$
for each $t$. Thus
$$\dif(\theta_t)_x(\tilde{X}_x-\tilde{Z}(t))=\dif(\theta_t)_x\tilde{X}_x-\tilde{X}_{\theta_t(x)}\in T_{\theta_t(x)}{\mathcal{F}}.$$
By construction, for $t\leq 0$ the flow $\theta_t$ is given by $h_{e^t}$ on $\mathrm{Tub}_{r'}(J)$, hence
$$\dif (h_\lambda)_x\tilde{X}_x-\tilde{X}_{h_\lambda(x)}\in T_{h_\lambda(x)}{\mathcal{F}},$$
for $\lambda\in(0,1]$.
\end{proof}

We can now prove that sections of $\mathscr{C}_{\mathcal{F}}$ admit linearized representatives.

\begin{proposition}[Homothetic invariance of the Molino sheaf]\label{proposition: invariance of molino sheaf under homothetic transformations}
Let $(M,{\mathcal{F}})$ be a complete singular Riemannian foliation with smooth Molino sheaf, $J\subset M$ an $\overline{{\mathcal{F}}}$-saturated homothety center, $h_\lambda:\mathrm{Tub}_{r}(J)\to \mathrm{Tub}_{\lambda r}(J)$ a homothetic transformation, and $P\subset J$ an open subset. Then any $X\in\mathscr{C}_{\mathcal{F}}(\mathrm{Tub}_r(P))$ admits a linearized representative, and hence is $h_\lambda$-invariant (in particular $\rho_P$-invariant).
\end{proposition}

\begin{proof}
Let $\tilde{X}$ be any representative for $X$ on $\mathrm{Tub}_r(P)$. By Lemma \ref{lemma: invariance of molino sheaf under homothetic transformations}, $\dif (h_\lambda)\tilde{X}-\tilde{X}\circ h_\lambda$ is tangent to ${\mathcal{F}}$, hence $\dif (h_{1/\lambda})\dif (h_\lambda)\tilde{X}-\dif (h_\lambda)\tilde{X}\circ h_\lambda$ is tangent to ${\mathcal{F}}$ as well. Then
$$\tilde{X}-\lim_{\lambda \to 0} \dif (h_\lambda)\tilde{X}\circ h_\lambda=\tilde{X}-\tilde{X}^\ell$$
must also be tangent to ${\mathcal{F}}$, and thus $\tilde{X}^\ell$ is a representative for $X$. We already know that $\tilde{X}^\ell$ is $h_\lambda$-invariant for any admissible $\lambda$.
\end{proof}

An $\overline{{\mathcal{F}}}$-saturated homothety center $J\subset M$ is $\mathfrak{a}$-invariant, since $\mathfrak{a}{\mathcal{F}}=\overline{{\mathcal{F}}}$. Notice that an ${\mathcal{F}}$-saturated open set is also $\mathfrak{a}$-invariant, since it must be $\overline{{\mathcal{F}}}$-saturated. By Proposition \ref{proposition: invariance of molino sheaf under homothetic transformations}, a homothetic retraction
$$h_J:U\times[0,1]\to U$$
is $\mathfrak{a}$-equivariant, for $U\subset\mathrm{Tub}_r(J)$ any $J$-star-shaped open set (in particular, for $U=\mathrm{Tub}_r(J)$). Let us sum this up in the following corollary.

\begin{corollary}\label{corollary: homothetic retractions are a-equivariant}
Let $(M,{\mathcal{F}})$ be a singular Killing foliation. A homothetic retraction $h_J:U\times[0,1]\to U$ onto an $\overline{{\mathcal{F}}}$-saturated homothety center $J\subset M$ is a foliate, $\mathfrak{a}$-equivariant deformation retraction of $U$ onto $J$.
\end{corollary}

It follows that the good cover $\mathcal{U}$ constructed in Proposition \ref{proposition: good cover} is \textit{equivariant}: any non-empty intersection $U$ of its elements admits a foliate, $\mathfrak{a}$-equivariant, deformation retraction onto the closure of some leaf $L_U\subset U$.

\subsection{Calculating tools}

Combining Corollary \ref{corollary: homothetic retractions are a-equivariant} and Corollary \ref{corollary: invariance of ebc under homotopies and deformation retracts}, we obtain the following property, which is crucial.

\begin{proposition}[Poincaré lemma for equivariant basic cohomology]\label{proposition: poincare lema for equivariant basic cohomology}
Let $(M,{\mathcal{F}})$ be a singular Killing foliation and suppose $J\subset M$ is an $\overline{{\mathcal{F}}}$-saturated homothety center. Then for any $r>0$ such that $\mathrm{Tub}_r(J)$ is a distinguished tubular neighborhood and any $J$-star-shaped open set $U\subset\mathrm{Tub}_r(J)$, the inclusion $J\to U$ induces an isomorphism
$$H_{\mathfrak{a}}({\mathcal{F}}{\vert}_U)\cong H_{\mathfrak{a}}({\mathcal{F}}{\vert}_J)$$
of $\sym(\mathfrak{a}^*)$-algebras.
\end{proposition}

Let us now establish that $H_{\mathfrak{a}}({\mathcal{F}})$ can be calculated via Mayer--Vietoris sequences. We will adapt the arguments from \cite{goertsches} for the regular case.

\begin{lemma}\label{lemma: basic forms are invariant}
Let ${\mathcal{F}}$ be a singular Killing foliation of $M$ with structural algebra $\mathfrak{a}$ and let $U\subset M$ be an open saturated subset. Then $\Omega({\mathcal{F}}{\vert}_U)^\mathfrak{a}=\Omega({\mathcal{F}}{\vert}_U)$, and hence $\Omega_\mathfrak{a}({\mathcal{F}}{\vert}_U)=\sym(\mathfrak{a}^*)\otimes\Omega({\mathcal{F}}|_U)$.
\end{lemma}

\begin{proof}
Since the regular locus $\Sigma_\mathrm{reg}$ of ${\mathcal{F}}$ is open and dense in $M$, it is sufficient to show that $\Omega({\mathcal{F}}{\vert}_{U\cap\Sigma_\mathrm{reg}})^\mathfrak{a}=\Omega({\mathcal{F}}{\vert}_{U\cap\Sigma_\mathrm{reg}})$. On $\Sigma_\mathrm{reg}$, the transverse $\mathfrak{a}$-action is nothing but the natural transverse action of the structural algebra of the regular foliation ${\mathcal{F}}_{\mathrm{reg}}$, so the result follows from \cite[Lemma 3.15]{goertsches}.
\end{proof}

The proof of \cite[Proposition 3.16]{goertsches} now adapts directly to the singular setting: for saturated open sets $U,V\subset M$,
$$0\longrightarrow \Omega({\mathcal{F}}{\vert}_{U\cup V})\stackrel{i_U^*\oplus i_V^*}{\longrightarrow}\Omega({\mathcal{F}}{\vert}_U)\oplus \Omega({\mathcal{F}}{\vert}_V)\stackrel{j_U^*-j_V^*}{\longrightarrow} \Omega({\mathcal{F}}{\vert}_{U\cap V})\longrightarrow 0$$
is a short exact sequence of $\mathfrak{a}^\star$-modules. Since tensoring with $\sym(\mathfrak{a}^*)$ preserves exactness, we obtain, by Lemma \ref{lemma: basic forms are invariant}, the short exact sequence
$$0\longrightarrow \Omega_\mathfrak{a}({\mathcal{F}}{\vert}_{U\cup V})\longrightarrow \Omega_\mathfrak{a}({\mathcal{F}}{\vert}_U)\oplus \Omega_\mathfrak{a}({\mathcal{F}}{\vert}_V)\longrightarrow \Omega_\mathfrak{a}({\mathcal{F}}{\vert}_{U\cap V})\longrightarrow 0$$
that induces an exact Mayer--Vietoris sequence in equivariant basic cohomology. Let us state this below.

\begin{proposition}[Mayer–Vietoris sequence for equivariant basic cohomology]
Let ${\mathcal{F}}$ be a singular Killing foliation of $M$ with structural algebra $\mathfrak{a}$ and $U,V\subset M$ saturated open sets. Then the sequence of $\sym(\mathfrak{a}^*)$-modules
$$\cdots\stackrel{\beta}{\longrightarrow}H_\mathfrak{a}({\mathcal{F}}{\vert}_{U\cup V})\stackrel{i_U^*\oplus i_V^*}{\longrightarrow}H_\mathfrak{a}({\mathcal{F}}{\vert}_U)\oplus H_\mathfrak{a}({\mathcal{F}}{\vert}_V)\stackrel{j_U^*-j_V^*}{\longrightarrow} H_\mathfrak{a}({\mathcal{F}}{\vert}_{U\cap V})\stackrel{\beta}{\longrightarrow}\cdots$$
is exact.
\end{proposition}

\subsection{Borel type localization}

A remarkable feature of the equivariant cohomology of the action of a torus $T$ on a manifold $M$ is that, modulo torsion, $H_T(M)$ can be recovered from the fixed point set $M^T$. This property became know as Borel's localization, as he observed close relations between $H_T(M)$ and $H_T(M^T)$ (see \cite[Chapter XII]{borel}) which later led to the formalization of the localization theorem. Our objective in this section is to obtain a version of this result for equivariant basic cohomology of singular Killing foliations. The case of regular Killing foliations, with the canonical action by the structural algebra, was treated in \cite{goertsches}. It was later generalized for arbitrary transverse actions on (regular) Riemannian foliations in \cite{lin}.

This topic is better formulated with the notion of localization of modules, for which we refer to \cite[Chapter 3]{atiyah2}. In fact we will be interested in the localization $S^{-1}H_{\mathfrak{a}}({\mathcal{F}})$ of $H_{\mathfrak{a}}({\mathcal{F}})$ with respect to $S=\sym(\mathfrak{a}^*)\setminus\{0\}$. We view an element of $S^{-1}H_{\mathfrak{a}}({\mathcal{F}})$ as a fraction $[\omega]/f$, for $[\omega]\in H_{\mathfrak{a}}({\mathcal{F}})$ and $f\in S$. The usual fraction operations turn $S^{-1}H_{\mathfrak{a}}({\mathcal{F}})$ into a vector space over the field $\mathrm{Q}(\mathfrak{a}^*)$ of fractions of $\sym(\mathfrak{a}^*)$, whose dimension is the rank of $H_{\mathfrak{a}}({\mathcal{F}})$. Recall that $[\omega]\in H_{\mathfrak{a}}({\mathcal{F}})$ is a \textit{torsion element} if there exists $f\in S$ such that $f[\omega]=0$. The submodule $\tor(H_{\mathfrak{a}}({\mathcal{F}}))$ of torsion elements is the \textit{torsion submodule} of $H_{\mathfrak{a}}({\mathcal{F}})$. A module is a \textit{torsion module} when it is equal to its torsion submodule. In particular, $H_{\mathfrak{a}}({\mathcal{F}})$ is torsion if, and only if, $S^{-1}H_{\mathfrak{a}}({\mathcal{F}})=0$. The following fact will be useful.

\begin{proposition}\label{proposition: cohomology of non closed leaves is torsion}
Let ${\mathcal{F}}$ be a singular Killing foliation. Then $H_{\mathfrak{a}}({\mathcal{F}}{\vert}_{\overline{L}})=\sym(\mathfrak{a}_L^*)$ for every $L\in{\mathcal{F}}$. In particular, if $L$ is not closed, then $H_{\mathfrak{a}}({\mathcal{F}}{\vert}_{\overline{L}})$ is a torsion module.
\end{proposition}

\begin{proof}
The restriction ${\mathcal{F}}{\vert}_{\overline{L}}$ is a (regular) Killing foliation of the manifold $\overline{L}$, hence it follows from \cite[Corollary 3.21]{goertsches} that $H_{\mathfrak{a}}({\mathcal{F}}{\vert}_{\overline{L}})=\sym(\mathfrak{a}_L^*)$, which is torsion when $L$ is not closed: if $f\in\sym(\mathfrak{a}^*)$ is any non-zero form that vanishes on $\mathfrak{a}_L^*$, then multiplication by $f$ is the zero map on $\sym(\mathfrak{a}_L^*)$.
\end{proof}

The regular version of Borel's localization theorem states that $S^{-1}H_{\mathfrak{a}}({\mathcal{F}})\cong S^{-1}H_{\mathfrak{a}}({\mathcal{F}}{\vert}_{M^\mathfrak{a}})$. In the singular setting there is a particular phenomenon that, in principle, prevents us from obtaining the same result for all Killing foliations: unlike the regular case, for some singular Killing foliations the closed locus $M^\mathfrak{a}$ is singular, that is, its connected components may not be submanifolds of $M$.

\begin{example}\label{example: propagation of non minimal closed leaves}
Fix $\lambda\in\mathbb{R}\setminus\{0\}$ and consider the isometric $\mathbb{R}$-action on $\mathbb{C}\times\mathbb{C}$ given by
$$t(z_1,z_2)=(e^{2\pi\mathrm{i}t}z_1,e^{2\pi\mathrm{i}\lambda t}z_2).$$
Denote by $H<\mathrm{SO}(4)$ the corresponding one-parameter subgroup. As we saw in Example \ref{example: homogeneous Riemannian foliations are Killing}, the homogeneous foliation ${\mathcal{F}}$ given by its orbits is Killing. Notice that its only singular leaf is the origin $\{0\}$. Since ${\mathcal{F}}$ is invariant by homotheties with respect to it, it is completely determined by its restriction to $\mathbb{S}^3\subset \mathbb{C}\times\mathbb{C}$, which is a regular Killing foliation with one dimensional leaves known as a \textit{generalized Hopf fibration}. We are interested in the case of an irrational $\lambda$, but it is instructive to first analyze rational cases in order to get a better grasp. If $\lambda=1$ one gets the usual Hopf fibration on $\mathbb{S}^3$, whereas if $\lambda=p/q\in\mathbb{Q}$, one gets its orbifold version, whose quotient is the weighted projective space $\mathbb{CP}(p,q)$ (also known as the $(p,q)$-football orbifold): a sphere with cone singularities of order $p$ and $q$ at the poles, corresponding to the exceptional orbits $\mathbb{R}(1,0)$ and $\mathbb{R}(0,1)$. In either case ${\mathcal{F}}$ is closed, so let us now assume that $\lambda$ is irrational. The exceptional orbits $\mathbb{R}(1,0)$ and $\mathbb{R}(0,1)$ remain closed (two linked circles), but now have infinite holonomy. As we saw in Example \ref{exe: molino sheaf of homogeneous singular foliation}, the closure $\overline{{\mathcal{F}}}$ is given by the orbits of $\overline{H}<\mathrm{SO}(4)$, which is a maximal torus. A regular $\overline{H}$-orbit is a torus in $\mathbb{S}^3$, on which ${\mathcal{F}}$ restricts to an irrational Kronecker foliation. By the homothetic invariance of ${\mathcal{F}}$, the closed locus $M^\mathfrak{a}$ is a double cone with vertex $\{0\}$ and with each nappe ``homothetically spanned'' by an exceptional orbit.
\end{example}

More generally, this phenomenon happens whenever there is a closed $l$-dimensional leaf $L$ near a stratum $\Sigma^k_\alpha$, with $k<l$: since $\mathscr{C}_{\mathcal{F}}$ is invariant by homothetic transformations (Proposition \ref{proposition: invariance of molino sheaf under homothetic transformations}), $L$ ``propagates'' towards $\Sigma^k_\alpha$, leading to singularities in $M^\mathfrak{a}$. This observation motivates the following definition.

\begin{definition}\label{definition: neat skf}
A singular Killing foliation $(M,{\mathcal{F}})$ is \textit{neat} when each connected component $C$ of $M^\mathfrak{a}$ is contained in a single stratum of ${\mathcal{F}}$.
\end{definition}

Notice that if ${\mathcal{F}}$ is neat then a connected component $C\subset M^{\mathfrak{a}}$ is a stratum of $\overline{{\mathcal{F}}}$. Also, if $\Sigma_\alpha$ is the ${\mathcal{F}}$-stratum containing $C$, then $C$ is a minimal stratum of $\overline{{\mathcal{F}}}{\vert}_{\Sigma_\alpha}$. It follows that each $C\subset M^\mathfrak{a}$ is a closed, saturated, horizontally totally geodesic submanifold of $M$. An important particular case is that of a singular Killing foliation ${\mathcal{F}}$ with $M^{\mathfrak{a}}\subset \Sigma_{\mathrm{min}}$, in which case $M^{\mathfrak{a}}$, if not empty, is precisely the minimal locus of $\overline{{\mathcal{F}}}$. Another subclass is that of singular Killing foliations with isolated closed leaves. Notice also that, of course, if $M^{\mathfrak{a}}=\emptyset$ or if ${\mathcal{F}}$ is regular, then ${\mathcal{F}}$ is automatically neat.

Let us now prove Borel's localization property for the equivariant basic cohomology of neat singular Killing foliations. We denote ${\mathcal{F}}^\mathfrak{a}={\mathcal{F}}{\vert}_{M^{\mathfrak{a}}}$.

\begin{theorem}[Borel localization for equivariant basic cohomology]\label{theorem: borel localization}
Let $(M,{\mathcal{F}})$ be a transversely compact, neat singular Killing foliation with structural algebra $\mathfrak{a}$, and let $i:M^{\mathfrak{a}}\to M$ be the natural inclusion. Then $\ker i^*=\tor(H_{\mathfrak{a}}({\mathcal{F}}))$ and $\coker i^*$ is a torsion module. Hence, the localized map
$$S^{-1}i^*:S^{-1}H_{\mathfrak{a}}({\mathcal{F}})\longrightarrow S^{-1}H_{\mathfrak{a}}({\mathcal{F}}^\mathfrak{a})$$
is an isomorphism, where $S=\sym(\mathfrak{a}^*)\setminus\{0\}$.
\end{theorem}

\begin{proof}
Let $M^{\mathfrak{a}}=C_1\sqcup\dots\sqcup C_k$ be the decomposition into connected components, which are finitely many since ${\mathcal{F}}$ is transversely compact, and choose pairwise disjoint distinguished tubular neighborhoods $\mathrm{Tub}_{r_j}(C_j)$. Take $\mathcal{U}=\{U_i\}$ a finite equivariant good cover for $(M,{\mathcal{F}})$ as in Proposition \ref{proposition: good cover} (see also Corollary \ref{corollary: homothetic retractions are a-equivariant}). Since each $C_j$ is a stratum of $\overline{{\mathcal{F}}}$, we have that if $U_i\cap C_j\neq \emptyset$, then $U_i$ is a tubular neighborhood around some leaf in $C_j$. Moreover we can suppose that in this case $U_i\subset \mathrm{Tub}_{r_j}(C_j)$, as we could have imposed this condition during the construction of $\mathcal{U}$. Define
$$U=\bigcup_{U_i\cap M^{\mathfrak{a}}\neq \emptyset} U_i.$$
Then the connected component of $U$ containing $C_j$ retracts by homotheties to it, so the inclusion $i{\vert}^U:M^\mathfrak{a}\to U$ induces an isomorphism $H_\mathfrak{a}({\mathcal{F}}{\vert}_U)\cong H_\mathfrak{a}({\mathcal{F}}^\mathfrak{a})$, by Proposition \ref{proposition: poincare lema for equivariant basic cohomology}. Define $V=\bigcup_{U_i\cap M^{\mathfrak{a}}= \emptyset} U_i$, so that $\{U,V\}$ is a cover of $M$ with the property that $\mathcal{U}$ restricts to finite equivariant good covers for both ${\mathcal{F}}{\vert}_V$ and ${\mathcal{F}}{\vert}_{U\cap V}$.

We can now follow the standard proofs of the classical Borel localization for torus actions (see, for instance, \cite[Theorem C.21]{guillemin3}, \cite[Theorem 11.4.4]{guillemin} and \cite[Theorem 8.1]{goertsches3}), which we include here for the sake of completeness. First, we claim that $H_{\mathfrak{a}}({\mathcal{F}}{\vert}_{V})$ and $H_{\mathfrak{a}}({\mathcal{F}}{\vert}_{U\cap V})$ are torsion modules. In fact, let us prove that if $W\subset M\setminus M^{\mathfrak{a}}$ is any saturated open set admitting a finite equivariant good cover $W=W_1\cup\dots\cup W_k$, then $H_{\mathfrak{a}}({\mathcal{F}}{\vert}_W)$ is torsion. If $k=1$ then $W$ retracts equivariantly to the closure of a non-closed leaf $L\subset V$, therefore $H_{\mathfrak{a}}({\mathcal{F}}{\vert}_{V})\cong H_{\mathfrak{a}}({\mathcal{F}}{\vert}_{\overline{L}})$ is torsion, by Proposition \ref{proposition: cohomology of non closed leaves is torsion} (and Corollary \ref{corollary: invariance of ebc under homotopies and deformation retracts}). For the induction step, define $W'=W_1\cup\dots\cup W_{k-1}$, so $W=W'\cup W_k$. Then, by the induction hypothesis, both extremes of the exact sequence
$$H_{\mathfrak{a}}({\mathcal{F}}{\vert}_{W'\cap W_k})\longrightarrow H_{\mathfrak{a}}({\mathcal{F}}{\vert}_W) \longrightarrow H_{\mathfrak{a}}({\mathcal{F}}{\vert}_W')\oplus H_{\mathfrak{a}}({\mathcal{F}}{\vert}_{W_k})$$
are torsion modules.

We have $i^*(\tor(H_\mathfrak{a}({\mathcal{F}})))\subset \tor(H_\mathfrak{a}({\mathcal{F}}^\mathfrak{a}))=0$, the last equality following from the fact that $H_{\mathfrak{a}}({\mathcal{F}}^\mathfrak{a})=\sym(\mathfrak{a}^*)\otimes H({\mathcal{F}}^\mathfrak{a})$ is free (see Example \ref{example: equivariant cohomology of trivial action}). Thus, $\tor(H_{\mathfrak{a}}({\mathcal{F}}))\subset \ker i^*$. Conversely, since $i=i_U\circ i{\vert}^U$, from the Mayer--Vietoris sequence for the cover $\{U,V\}$ we obtain the exact sequence
\begin{equation}\label{equation: M-V sequence for the cover UV}
\cdots\longrightarrow H_\mathfrak{a}({\mathcal{F}})\stackrel{i^*\oplus i_V^*}{\longrightarrow}H_\mathfrak{a}({\mathcal{F}}^\mathfrak{a})\oplus H_\mathfrak{a}({\mathcal{F}}{\vert}_V) \stackrel{j^*-j_V^*}{\longrightarrow} H_\mathfrak{a}({\mathcal{F}}{\vert}_{U\cap V})\longrightarrow\cdots.
\end{equation}
One has $\ker(i^*\oplus i_V^*)=\ker i^*\cap\ker i_V^*$, so the sequence
$$H_\mathfrak{a}({\mathcal{F}}{\vert}_{U\cap V})\longrightarrow \ker i^*\stackrel{i_V^*}{\longrightarrow} H_\mathfrak{a}({\mathcal{F}}{\vert}_V)$$
of $\sym(\mathfrak{a}^*)$-modules is also exact. As the extremes of this sequence are torsion modules, it follows that $\ker i^*\subset \tor(H_{\mathfrak{a}}({\mathcal{F}}))$.

The exactness of \eqref{equation: M-V sequence for the cover UV} also implies that
\begin{align*}\coker(i^*\oplus i_V^*)  = &\frac{H_\mathfrak{a}({\mathcal{F}}^\mathfrak{a})\oplus H_\mathfrak{a}({\mathcal{F}}{\vert}_V)}{\image(i^*\oplus i_V^*)}=\frac{H_\mathfrak{a}({\mathcal{F}}^\mathfrak{a})\oplus H_\mathfrak{a}({\mathcal{F}}{\vert}_V)}{\ker(j^*-j_V^*)}\\
\cong & \image(j^*-j_V^*)\subset H_\mathfrak{a}({\mathcal{F}}{\vert}_{U\cap V}),
\end{align*}
so $\coker(i^*\oplus i_V^*)$ is a torsion module. The projection $H_\mathfrak{a}({\mathcal{F}}^\mathfrak{a})\oplus H_\mathfrak{a}({\mathcal{F}}{\vert}_V)\to H_\mathfrak{a}({\mathcal{F}}^\mathfrak{a})$ induces a surjective map $\coker(i^*\oplus i_V^*)\to \coker i^*$, so $\coker i^*$ is torsion as well.

Finally, since localization preserves exactness \cite[Proposition 3.3]{atiyah2} and kills torsion, from the exact sequence
$$0 \longrightarrow \ker i^* \longrightarrow H_\mathfrak{a}({\mathcal{F}})\stackrel{i^*}{\longrightarrow}H_\mathfrak{a}({\mathcal{F}}^\mathfrak{a})\longrightarrow \coker i^* \longrightarrow 0$$
it follows that $S^{-1}i^*:S^{-1}H_{\mathfrak{a}}({\mathcal{F}})\to S^{-1}H_{\mathfrak{a}}({\mathcal{F}}^\mathfrak{a})$ is an isomorphism.
\end{proof}

The following can be seen as a transverse version of \cite[Corollary 5.4]{goertsches} (notice, although, that we cannot obtain a complete characterization of the presence of closed leaves in terms of $H_{\mathfrak{a}}({\mathcal{F}})$, because of the neatness hypothesis).

\begin{corollary}
Let $(M,{\mathcal{F}})$ be a transversely compact, singular Killing foliation.
\begin{enumerate}
\item If $H_{\mathfrak{a}}({\mathcal{F}})$ is a torsion module, then either $\f$ has no closed leaves or it is not neat. 
\item If $\f$ has no closed leaves, then $H_{\mathfrak{a}}({\mathcal{F}})$ is a torsion module.
\end{enumerate}
\end{corollary}

\begin{proof}
Suppose there is a closed leaf $L\in{\mathcal{F}}$ and $\f$ is neat. Then $i^*(1)=1\in H_{\mathfrak{a}}({\mathcal{F}}^\mathfrak{a})$ and, since $H_{\mathfrak{a}}({\mathcal{F}}^\mathfrak{a})\cong \sym(\mathfrak{a}^*)\otimes H({\mathcal{F}}^\mathfrak{a})$ is torsion-free, $1\in H_{\mathfrak{a}}({\mathcal{F}})$ must have no torsion.

For the second item, if there are no closed leaves, that is, $M^{\mathfrak{a}}=\emptyset$, then $H_{\mathfrak{a}}({\mathcal{F}}^\mathfrak{a})=0$. Moreover, in this case $\f$ is neat, so Theorem \ref{theorem: borel localization} applies, giving us that $H_{\mathfrak{a}}({\mathcal{F}})$ is torsion.
\end{proof}

\section{Applications of the localization theorem}

In this section we will see that there is a spectral sequence that converges to $H_{\mathfrak{a}}({\mathcal{F}})$. This can then be used in association with Theorem \ref{theorem: borel localization} to show that $\chi({\mathcal{F}})=\chi({\mathcal{F}}^\mathfrak{a})$, in analogy with the localization of the Euler characteristic of a manifold to the fixed point set of a torus action. In fact, all constructions will be direct adaptations of the classical case of torus actions on manifolds, as presented in \cite{goertsches3} (see also \cite[Chapter 6]{guillemin}), so we will not delve into too much details of the proofs. These results are also already established for the case or regular Killing foliations in \cite{goertsches}.

\subsection{A spectral sequence for equivariant basic cohomology}

Let ${\mathcal{F}}$ be a singular Killing foliation of $M$. Endow $C_{\mathfrak{a}}({\mathcal{F}})$ with the bigrading
$$C^{k,l}_{\mathfrak{a}}({\mathcal{F}})=\begin{cases}
(\sym^{k/2}(\mathfrak{a}^*)\otimes \Omega^l({\mathcal{F}}))^\mathfrak{a}, & \mbox{if } k \mbox{ is even},\\
0, & \mbox{if } k \mbox{ is odd},\end{cases}
$$
so that $C^n_{\mathfrak{a}}({\mathcal{F}})=\bigoplus_{k+l=n} C^{k,l}_{\mathfrak{a}}({\mathcal{F}})$ recovers the grading \eqref{grading of the Cartan complex}. The differential $d_\mathfrak{a}$ splits as $1\otimes d +\delta$, where $(\delta \omega)(X)=-\iota_X(\omega(X))$. The filtration
$$F^kC_{\mathfrak{a}}({\mathcal{F}})=\bigoplus_{\substack{j\geq k\\ l\geq 0}} C^{j,l}_{\mathfrak{a}}({\mathcal{F}}),$$
is canonically bounded, hence the associated spectral sequence $(E_r)$ is first quadrant and converges, that is,
$$E_\infty\cong H_\mathfrak{a}({\mathcal{F}})$$
as vector spaces. By Lemma \ref{lemma: basic forms are invariant} we have $\Omega_\mathfrak{a}({\mathcal{F}})=\sym(\mathfrak{a}^*)\otimes\Omega({\mathcal{F}})$, which leads to
$$E_1\cong \sym(\mathfrak{a}^*)\otimes H({\mathcal{F}}).$$

Each page $E_r=H(E_{r-1},d_{r-1})$ inherits a $\sym(\mathfrak{a}^*)$-module structure recursively from $E_{r-1}$. On $E_1\cong \sym(\mathfrak{a}^*)\otimes H({\mathcal{F}})$ it is given by the usual product on the first factor. If ${\mathcal{F}}$ is transversely compact, by Theorem \ref{theorem: basic cohomology is finite dimensional} we have $\dim H({\mathcal{F}})<\infty$, hence $E_1$ is finitely generated as an $\sym(\mathfrak{a}^*)$-module. Since $\sym(\mathfrak{a}^*)$ is Noetherian, submodules and quotients of ﬁnitely generated $\sym(\mathfrak{a}^*)$-modules are ﬁnitely generated, thus each $E_r$ is ﬁnitely generated. In particular, let $\{x_i\}\subset E_\infty$ be a finite set of homogeneous generators for $E_\infty$. Then one proves that representatives $y_i\in H_\mathfrak{a}({\mathcal{F}})$ chosen via
$$E_\infty^{k,l}\cong \frac{F^kH^{k+l}_\mathfrak{a}({\mathcal{F}})}{F^{k+1}H^{k+l}_\mathfrak{a}({\mathcal{F}})}$$
generate $H_\mathfrak{a}({\mathcal{F}})$ (see \cite[Lemma A.17]{goertsches3}). Summing up (cf. \cite[Lemma A.18]{goertsches3}):

\begin{proposition} Let ${\mathcal{F}}$ be a transversely compact singular Killing foliation. Then $H_\mathfrak{a}({\mathcal{F}})$ is finitely generated as an $\sym(\mathfrak{a}^*)$-module.
\end{proposition}

As a corollary, it follows that $\rank H_\mathfrak{a}^{\mathrm{even}}({\mathcal{F}})=\rank E_\infty^{\mathrm{even}}$ and $\rank H_\mathfrak{a}^{\mathrm{odd}}({\mathcal{F}})=\rank E_\infty^{\mathrm{odd}}$, with respect to the $\mathbb{Z}_2$-graded decompositions in even and odd degree elements of $H_\mathfrak{a}({\mathcal{F}})$ and $E_\infty$ (see \cite[Corollary A.19]{goertsches3}). This is proved by comparing the respective Poincaré series, and using that $E_\infty\cong H_\mathfrak{a}({\mathcal{F}})$ as vector spaces.

\subsection{Localization of the basic Euler characteristic} We can now directly adapt \cite[Theorem 9.3]{goertsches3} to our transverse setting.

\begin{theorem}[Localization of the basic Euler characteristic]\label{theorem: localization of bec}
Let ${\mathcal{F}}$ be a transversely compact, neat singular Killing foliation with structural algebra $\mathfrak{a}$. Then
$$\chi({\mathcal{F}})=\chi({\mathcal{F}}^\mathfrak{a}).$$
\end{theorem}

\begin{proof}
We have
\begin{align*}
\chi({\mathcal{F}}) & = \dim H^{\mathrm{even}}({\mathcal{F}}) - \dim H^{\mathrm{odd}}({\mathcal{F}})\\
 & =  \dim_{\mathrm{Q}(\mathfrak{a}^*)} \mathrm{Q}(\mathfrak{a}^*)\otimes H^{\mathrm{even}}({\mathcal{F}}) - \dim_{\mathrm{Q}(\mathfrak{a}^*)} \mathrm{Q}(\mathfrak{a}^*)\otimes H^{\mathrm{odd}}({\mathcal{F}}).
\end{align*}
Using that $E_1\cong \sym(\mathfrak{a}^*)\otimes H({\mathcal{F}})$, this leads to
$$\chi({\mathcal{F}})=\rank E_1^{\mathrm{even}}-\rank E_1^{\mathrm{odd}}.$$

Localizing each page of the spectral sequence $(E_r)$, we obtain $\mathbb{Z}_2$-graded vector spaces with odd differentials $d_r:E_r\to E_r$, so that $E_{r+1}=H(E_r,d_r)$. It follows that $\rank E_i^{\mathrm{even}}-\rank E_i^{\mathrm{odd}}=\rank E_{i+1}^{\mathrm{even}}-\rank E_{i+1}^{\mathrm{odd}}$, since the Euler characteristic is preserved under taking cohomology (see \cite[Lemma 9.2]{goertsches3}). Therefore
\begin{align*}
\chi({\mathcal{F}}) & =\rank E_\infty^{\mathrm{even}}-\rank E_\infty^{\mathrm{odd}}\\
 & = \rank H_\mathfrak{a}^{\mathrm{even}}({\mathcal{F}}) - \rank H_\mathfrak{a}^{\mathrm{odd}}({\mathcal{F}}).
\end{align*}

Using Theorem \ref{theorem: borel localization} and the fact that $H_{\mathfrak{a}}({\mathcal{F}}^\mathfrak{a})=\sym(\mathfrak{a}^*)\otimes H({\mathcal{F}}^\mathfrak{a})$, we then have
\begin{align*}
\chi({\mathcal{F}}) & = \rank H_\mathfrak{a}^{\mathrm{even}}({\mathcal{F}}^\mathfrak{a}) - \rank H_\mathfrak{a}^{\mathrm{odd}}({\mathcal{F}}^\mathfrak{a})\\
 & = \dim_{\mathrm{Q}(\mathfrak{a}^*)} \mathrm{Q}(\mathfrak{a}^*)\otimes H^{\mathrm{even}}({\mathcal{F}}^\mathfrak{a}) - \dim_{\mathrm{Q}(\mathfrak{a}^*)} \mathrm{Q}(\mathfrak{a}^*)\otimes H^{\mathrm{odd}}({\mathcal{F}}^\mathfrak{a}) \\
 & = \dim H^{\mathrm{even}}({\mathcal{F}}^\mathfrak{a}) - H^{\mathrm{odd}}({\mathcal{F}}^\mathfrak{a})\\
 & = \chi({\mathcal{F}}^\mathfrak{a}),
\end{align*}
which finishes the proof.
\end{proof}

The following are obvious consequences.

\begin{corollary}\label{corollary: chi detects closed leaves}
Let ${\mathcal{F}}$ be a transversely compact, singular Killing foliation. If $\chi({\mathcal{F}})\neq0$, then $\f$ has at least one closed leaf.
\end{corollary}

\begin{proof} If $\f$ has no closed leaves (and hence is neat), then $\chi({\mathcal{F}})=\chi({\mathcal{F}}^\mathfrak{a})=0$.
\end{proof}

\begin{corollary}\label{corollary: localization of bec finite case}
Let ${\mathcal{F}}$ be a transversely compact, singular Killing foliation whose closed leaves are isolated. Then
$$\#{\mathcal{F}}^\mathfrak{a}=\chi({\mathcal{F}}),$$
that is, the basic Euler characteristic is precisely the number of closed leaves.
\end{corollary}

\subsection{The equivariantly formal case}

As in the classical case of Lie group actions, one can consider the equivariantly formal case, which includes the equivariant cohomology of compact symplectic manifolds with Hamiltonian torus actions. In our case, we will say that singular Killing foliation ${\mathcal{F}}$ is \textit{equivariantly formal} if its equivariant basic cohomology is equivariantly formal, that is,
$$H_\mathfrak{a}({\mathcal{F}})\cong \sym(\mathfrak{a}^*)\otimes H({\mathcal{F}})$$
as $\sym(\mathfrak{a}^*)$-modules (cf. \cite[Section 6.3]{goertsches}). Equivalently, ${\mathcal{F}}$ is equivariantly formal when $(E_r)$ collapses at the first page (see \cite[Theorem 7.3]{goertsches3}). This is the case, for instance, if $H^{\mathrm{odd}}({\mathcal{F}})=0$ (see \cite[Corollary A.10]{goertsches3}). The following theorem is a singular version of \cite[Theorem 1]{goertsches}.

\begin{theorem}\label{theorem: comparison dimensions}
Let ${\mathcal{F}}$ be a transversely compact, neat singular Killing foliation with structural algebra $\mathfrak{a}$. Then
$$\dim H({\mathcal{F}}^\mathfrak{a})\leq \dim H({\mathcal{F}}),$$
and equality holds if, and only if, ${\mathcal{F}}$ is equivariantly formal.
\end{theorem}

\begin{proof}
The proof of the classical case of torus actions in \cite[Proposition 9.6]{goertsches3} adapts directly. One has
\begin{align*}
H({\mathcal{F}}^\mathfrak{a}) & = \rank(\sym(\mathfrak{a}^*)\otimes H({\mathcal{F}}^\mathfrak{a}))\\
 & = \rank H_\mathfrak{a}({\mathcal{F}}^\mathfrak{a})\\
 & = \rank H_\mathfrak{a}({\mathcal{F}})\\
 & \leq \dim H({\mathcal{F}}),
\end{align*}
The last inequality comes from $E_1\cong \sym(\mathfrak{a}^*)\otimes H({\mathcal{F}})$. It is clear that equality holds if ${\mathcal{F}}$ is equivariantly formal. Conversely, localizing each page of $(E_r)$ one sees that, if some $d_r$ is not zero, there must be a drop in the dimension at the corresponding pages. So if $\dim H({\mathcal{F}}^\mathfrak{a})= \dim H({\mathcal{F}})$ (hence $\rank H_\mathfrak{a}({\mathcal{F}})=\rank E_1$), then $d_r=0$ for all $r\geq1$, and the sequence collapses at the first page.
\end{proof}

\subsection*{Acknowledgements}
I am grateful to Professors M.~Alexandrino and D.~Töben for many helpful discussions.


\begin{thebibliography}{10}

\bibitem{alex2} M. Alexandrino, F. Caramello: \textit{Leaf closures of Riemannian foliations: a survey on topological and geometric aspects of Killing foliations}, Expo. Math. \textbf{40} (2022), 177--230, {\url{https://doi.org/10.1016/j.exmath.2021.11.00}}.

\bibitem{alex4} M. Alexandrino, M. Radeschi: \textit{Closure of singular foliations: the proof of Molino's conjecture}, Compos. Math. \textbf{153} (2017), 2577--2590, {\url{http://doi.org/10.1112/s0010437x17007485}}

\bibitem{alex9}  M. Alexandrino, M. Radeschi: \textit{Smoothness of isometric flows on orbit spaces and applications to the theory of foliations}, Transform. Groups \textbf{22}(1) (2017), 1--27, {\url{https://doi.org/10.1007/s00031-016-9386-5}}.

\bibitem{atiyah2} M. Atiyah, I. MacDonald: \textit{Introduction to Commutative Algebra}, Addison-Wesley Publishing Co., 1969.

\bibitem{borel} A. Borel: \textit{Seminar on Transformation Groups}, Princeton University Press, 1960.

\bibitem{caramello} F. Caramello, D. Töben: \textit{Positively curved Killing foliations via deformations}, Trans. Amer. Math. Soc. \textbf{372} (2019), 8131--8158, {\url{https://doi.org/10.1090/tran/7893}}.

\bibitem{caramello2} F. Caramello, D. Töben: \textit{Equivariant basic cohomology under deformations}, Math. Z. \textbf{299} (2021), 2461--2482, {\url{https://doi.org/10.1007/s00209-021-02768-w}}.

\bibitem{choi} S. Choi: \textit{Geometric Structures on 2-orbifolds: Exploration of Discrete Symmetry}, MSJ Memoirs \textbf{27}, Mathematical Society of Japan, 2012.

\bibitem{ghys2} E. Ghys: \textit{Une feuilletage analytique dont la cohomologie basique est de dimension infinie}, Publ. de l'IRMA de Lille \textbf{7}, 1985.

\bibitem{goertsches} O. Goertsches, D. Töben: \textit{Equivariant basic cohomology of Riemannian foliations}, J. Reine Angew. Math. \textbf{2018}(745) (2018), 1--40, {\url{https://doi.org/10.1515/crelle-2015-0102}}.

\bibitem{goertsches2} O. Goertsches, H. Nozawa, D. Töben: \textit{Localization of Chern--Simons type invariants of Riemannian foliations}, Isr. J. Math. \textbf{222} (2017), 867--920, {\url{https://doi.org/10.1007/s11856-017-1608-6}}.

\bibitem{goertsches3} O. Goertsches, L. Zoller: \textit{Equivariant  de Rham cohomology: theory and applications}, São Paulo J. Math. Sci. \textbf{13} (2019), 539--596, {\url{https://doi.org/10.1007/s40863-019-00129-4}}.

\bibitem{guillemin3} V. Guillemin, V. Ginzburg, Y. Karshon: \textit{Moment maps, cobordisms, and Hamiltonian group actions}, Mathematical Surveys and Monographs \textbf{98}, American Mathematical Society, 2002.

\bibitem{guillemin}  V. Guillemin, S. Sternberg: \textit{Supersymmetry and Equivariant de Rham Theory}, Springer-Verlag, 1999.

\bibitem{guillemin2} V. Guillemin, S. Sternberg: \textit{Symplectic Techniques in Physics}, Cambridge University Press, 1984.

\bibitem{inagaki} M. Inagaki: \textit{Um modelo semi-local para folheações Riemannianas singulares} (Ph.D. Thesis), Universidade de São Paulo, 2020, {\url{https://doi.org/10.11606/T.45.2020.tde-07012021-210350}}.

\bibitem{kacimi} A. Kacimi-Alaoui, V. Sergiescu, G. Hector: \textit{La cohomologie basique d'un feuilletage Riemannien est de dimension finie}, Math. Z. \textbf{188} (1985), 593--599, {\url{https://doi.org/10.1007/BF01161658}}.

\bibitem{lin} Y. Lin, R. Sjamaar, \textit{Cohomological localization for transverse Lie algebra actions on Riemannian foliations}, J. Geom. Phys. \textbf{158} (2020), 103887, {\url{https://doi.org/10.1016/j.geomphys.2020.103887}}.

\bibitem{lytchak} A. Lytchak, M. Radeschi: \textit{Algebraic nature of singular Riemannian foliations in spheres}, J. Reine Angew. Math. \textbf{2018}(744) (2018), 265--273, {\url{https://doi.org/10.1515/crelle-2016-0010}}.

\bibitem{mendes} R. Mendes, M. Radeschi: \textit{A slice theorem for singular Riemannian foliations, with applications}, Trans. Amer. Math. Soc. \textbf{371} (2019), 4931--4949, {\url{https://doi.org/10.1090/tran/7502}}.

\bibitem{molino} P. Molino: \textit{Riemannian Foliations}, Progress in Mathematics \textbf{73}, Birkhäuser, 1988.

\bibitem{molino2} P. Molino: \textit{Orbit-like foliations}, in: T. Mizutani, K. Masuda, S. Matsumoto, T. Inaba, T. Tsuboi, Y. Mitsumatsu (Eds.), \textit{Geometric Study of Foliations}, World Scientific, 1993, 97--119.

\bibitem{mozgawa} W. Mozgawa: \textit{Feuilletages de Killing}, Collect. Math. \textbf{36}(3) (1985), 285--290.

\bibitem{schwarz} G. Schwarz: \textit{Lifting smooth homotopies of orbit spaces}, Publ. Math. IHÉS \textbf{51} (1980), 37--135.

\bibitem{toben} D. Töben: \textit{Localization of basic characteristic classes}, Ann. Inst. Fourier (Grenoble), \textbf{64}(2) (2014), 537--570, {\url{https://doi.org/10.5802/aif.2857}}.

\bibitem{thorbergsson} G. Thorbergsson: \textit{Singular Riemannian Foliations and Isoparametric Submanifolds}, Milan J. Math. \textbf{78} (2010), 355--370, {\url{https://doi.org/10.1007/s00032-010-0112-9}}.


\bibitem{wilking} B. Wilking: \textit{A duality theorem for Riemannian foliations in nonnegative sectional curvature}, Geom. Funct. Anal. \textbf{17}(4) (2007), 1297–1320 {\url{https://doi.org/10.1007/s00039-007-0620-0}}.

\bibitem{wolak} R. Wolak: \textit{Basic cohomology for singular Riemannian foliations}, Mon. Hefte Math. \textbf{128} (1999), 159--163, {\url{https://doi.org/10.1007/s006050050053}}.

\end{thebibliography}
\end{document}